\documentclass[hidelinks,onefignum,onetabnum]{siamart251216}

\usepackage{cite}
\usepackage{array}

\usepackage{lipsum}
\usepackage{amsfonts}
\usepackage{graphicx}
\usepackage{epstopdf}
\usepackage{algorithmic}
\ifpdf
  \DeclareGraphicsExtensions{.eps,.pdf,.png,.jpg}
\else
  \DeclareGraphicsExtensions{.eps}
\fi

\newsiamremark{remark}{Remark}
\newsiamremark{hypothesis}{Hypothesis}
\crefname{hypothesis}{Hypothesis}{Hypotheses}
\newsiamthm{claim}{Claim}
\newsiamremark{fact}{Fact}
\crefname{fact}{Fact}{Facts}

\headers{Well-balanced and AP IMEX FV Schemes}{Artiano, Ranocha, and Samantaray}

\title{Asymptotic-Preserving and Well-Balanced Linearly Implicit IMEX Schemes for the Anelastic Limit of the Isentropic Euler Equations with Gravity\thanks{Submitted to the editors DATE.
\funding{This work was funded by the Deutsche Forschungsgemeinschaft
(DFG, German Research Foundation, project number 528753982
as well as within the DFG priority program SPP~2410 with project number 526031774).
We acknowledge support from the Mainz Institute of Multiscale Modeling (M3ODEL).
This work was supported by the Max Planck Graduate Center with the
Johannes Gutenberg University of Mainz (MPGC).}}}

\author{Marco Artiano\thanks{Institute of Mathematics, Johannes Gutenberg University Mainz, Germany (\email{martiano@uni-mainz.de}).}
\and Hendrik Ranocha\thanks{Institute of Mathematics, Johannes Gutenberg University Mainz, Germany (\email{hendrik.ranocha@uni-mainz.de}).}
\and Saurav Samantaray \thanks{Institute of Mathematics, Johannes Gutenberg University Mainz, Germany (\email{ssamanta@uni-mainz.de}).}}

\usepackage{amsopn}

\newcommand{\D}{\partial}

\newcommand{\Dlt}{\Delta t}

\newcommand{\Dt}{\partial_t}
\newcommand{\dvg}{\nabla \cdot }

\newcommand{\grd}{\nabla}

\newcommand{\mbb}{\mathbb}

\newcommand{\mcal}{\mathcal}

\newcommand{\Norm}[1]{{\left\vert\kern-0.25ex\left\vert\kern-0.25ex\left\vert #1
    \right\vert\kern-0.25ex\right\vert\kern-0.25ex\right\vert}}

\newcommand{\veps}{\varepsilon}
\newcommand{\vx}{\mathbf{x}}
\newcommand{\vq}{\mathbf{q}}
\newcommand{\vu}{\mathbf{u}}
\newcommand{\mco}{\mathcal{O}}
\newcommand{\pdxo}{\partial_{x_1}}
\newcommand{\mcst}{\mathcal{\tilde{S}}}

\ifpdf
\hypersetup{
  pdftitle={Asymptotic-Preserving and Well-Balanced Linearly Implicit IMEX Schemes for the Anelastic Limit of the Isentropic Euler Equations with Gravity},
  pdfauthor={Artiano, Ranocha, and Samantaray}
}
\fi

\begin{document}

\maketitle

% REQUIRED
\begin{abstract}
We consider the compressible Euler system with anelastic scaling, modeling isentropic flows under the influence of gravity.
In the zero-Mach-number limit, the solution of the compressible Euler system converges to a variable density anelastic incompressible limit system.
In this work, we present the design and analysis of a class of higher-order linearly implicit IMEX Runge-Kutta schemes that are asymptotic preserving, i.e., they respect the transitory nature of the governing equations in the limit.
The presence of gravitational potential warrants the incorporation of the well-balancing property.
The scheme is developed as a novel combination of a penalization of a linear steady state, a finite-volume balance-preserving reconstruction, and a source term discretization preserving steady states.
The penalization plays a crucial role in obtaining a linearly implicit scheme, and well-balanced flux-source discretization ensures accuracy in very low Mach number regimes.
Some results of numerical case studies are presented to corroborate the theoretical assertions.
\end{abstract}

% REQUIRED

% REQUIRED

\section{Introduction}
\label{sec:intro}
Numerical approximation of the compressible Euler equations, owing to the multi-scale nature of the flow dynamics they model, is a challenging task. These equations model flows with velocities ranging from supersonic to low-speed (atmospheric) flows.
In this work, we consider the isentropic Euler equations with periodic or reflecting boundary conditions and source terms describing the effect of a gravitational potential, i.e.,
\begin{align}
  \D_t \rho + \grd \cdot ( \rho \vu) &= 0, \label{eq:iseuler_mass_ane} \\
  \D_t (\rho \vu ) + \grd \cdot (\rho \vu \otimes \vu ) + \frac{\grd p}
  {\veps^2} & = -\frac{\rho \grd \phi}{\veps^2},  \label{eq:iseuler_mom_ane}
\end{align}
where $\rho(t,\vx) \in \mathbb{R}$ is the fluid density, $\vq (t,\vx) = \rho \vu \in \mathbb{R}^d$ is the momentum, $p(t,\vx)  \in \mathbb{R}$ is the pressure, $\phi(\vx) \in \mathbb{R}$ is a known gravitational potential, $t$ denotes time, and $\vx \in \mathbb{R}^d$ is the spatial variable in $d$ dimensions.
The parameter $\veps$ is a non-dimensional scaled Mach number describing the ratio of a reference fluid velocity to that of a reference sound velocity, i.e.,
\[\veps = \frac{u_\mathrm{ref}}{\sqrt{p^{\prime}(\rho_\mathrm{ref})}}.\]
The above system is closed by considering an isentropic equation of state (EOS), which links the pressure and the density variables by the following power law:
\begin{equation}
  \label{eq:EOS}
  p=P(\rho) = \rho^{\gamma},
\end{equation}
with the isentropic constant $\gamma>1$.
The $\veps^2$ denominator of the source term in~\eqref{eq:iseuler_mom_ane} corresponds to a particular choice of scaling called anelastic scaling~\cite{FN09, Kle09, OP62}.

The homogeneous counterpart of the above system has been extensively studied in the literature, e.g., \cite{KM81}.
In \cite{FKM19} the authors have investigated the homogeneous system with so-called well-prepared and ill-prepared initial conditions using the framework of dissipative measure-valued solutions to conclude that the solutions of the compressible system converge to smooth solutions of the incompressible Euler system when the Mach number goes to zero.

A formal derivation of the anelastic model for the adiabatic pressure law is carried out in \cite{Kle09} and for an isentropic pressure law in \cite{AS20_AIMS}.
Following \cite{Kle09}, formally, to obtain the limit system for the compressible Euler system \eqref{eq:iseuler_mass_ane}--\eqref{eq:iseuler_mom_ane} when $\veps \to 0$, we begin by assuming that the dependent variables admit the three-term ansatz
\begin{equation}\label{eq:ansatz}
    f(t,\vx) = f_{(0)}(t, \vx) + \veps f_{(1)} (t, \vx) + \veps^2 f_{(2)}(t,\vx) + \mco (\veps^3).
\end{equation}
Plugging the above into the Euler system \eqref{eq:iseuler_mass_ane}--\eqref{eq:iseuler_mom_ane} and first balancing the $\mco (\veps^{-2})$ terms in the momentum equation \eqref{eq:iseuler_mom_ane} yields the hydrostatic balance at zeroth order:
\begin{equation}
  \label{eq:hydro_bal_an}
  \nabla p_{(0)}=-\rho_{(0)} \grd\phi.
\end{equation}
Making the anelastic assumption that the leading-order density $\rho_{(0)}$ is temporally constant, so that any temporal variation in density occurs only at higher orders of $\veps$, we can obtain the anelastic density for the limit system as
\begin{equation}\label{eq:rho_eq}
    \rho_\mathrm{eq} : = \rho_{(0)}(\vx) = \left(C_0 - \frac{\gamma - 1}{\gamma} \phi(\vx) \right)^{\frac{1}{\gamma - 1}},
\end{equation}
where $C_0$ is a constant of integration.
Similarly, a balance between $p_{(1)}$ and $\rho_{(1)}$ like  \eqref{eq:hydro_bal_an} can be obtained when $\mco (\veps^{-1})$ terms are considered from the momentum balance \eqref{eq:iseuler_mom_ane}. So, both the zeroth-order and the first-order densities $\rho_{(0)}$ and $\rho_{(1)}$, respectively, exhibit  the same dynamics.
Thus, we absorb one into the other and get the final multi-scale expansion for the density:
\begin{equation}
    \rho(t, \vx) = \rho_\mathrm{eq}(\vx) + \veps^2 \rho_{(2)}(t,\vx).
\end{equation}
Considering the $\mco (\veps^0)$ terms from the Euler system \eqref{eq:iseuler_mass_ane}--\eqref{eq:iseuler_mom_ane} we obtain the anelastic limit system:
\begin{align}
  \grd \cdot \left( \rho_{(0)} \vu_{(0)}\right) &=0, \label{eq:ane_lim_div} \\
  \D_t \left(\rho_{(0)} \vu_{(0)}\right) + \grd \cdot \left(\rho_{(0)} \vu_{(0)} \otimes
  \vu_{(0)} \right) + \grd p_{(2)} & = -\rho_{(2)} \nabla\phi.  \label{eq:ane_lim_mom}
\end{align}
The asymptotic expansion for the pressure using the equation of state $p = P(\rho)$ gives
\begin{equation*}
    p_{(0)} + \veps^2 p_{(2)}=  P \left(\rho_{(0)} +
       \veps^2 \rho_{(2)}\right)
    = P(\rho_{(0)}) + \veps^2
     \rho_{(2)} P^{\prime} (\rho_{(0)}).
\end{equation*}
We close the anelastic limit system \eqref{eq:ane_lim_div}--\eqref{eq:ane_lim_mom} by comparing the $\mco (\veps^2)$ terms in the above expansion to obtain the closure relation
\begin{equation}
    p_{(2)} = \rho_{(2)} P^{\prime} (\rho_{(0)}).
\end{equation}
Note that, as explored in \cite{FKM19}, the convergence of the solution of the compressible system \eqref{eq:iseuler_mass_ane}--\eqref{eq:iseuler_mom_ane} to the anelastic limit system \eqref{eq:ane_lim_div}--\eqref{eq:ane_lim_mom} can be obtained for both well-prepared and ill-prepared initial data; this notion will be defined precisely later.

Explicit time discretizations are plagued by prohibitively restrictive CFL conditions since the stability constraints are directly proportional to $\veps$ in the low-Mach regime.
For $\veps \approx 0$ these schemes, while adhering to the stability requirements, become almost non-evolving in time due to rather small time steps.
Moreover, explicit schemes also suffer from inaccuracies in the low Mach number regime due to their inability to filter out acoustic waves \cite{Del10}. The framework of asymptotic-preserving (AP) schemes, initially introduced by Jin \cite{ap_jin}, takes into account the transitional behavior of the governing equations for $\veps \to 0$.
These schemes transform into approximations of the limit system while remaining stable with their stability requirements being independent of the singular perturbation parameter $\veps$ as $\veps \to 0$.
Semi-implicit time-stepping schemes provide an efficacious framework for designing AP schemes \cite{DT11, DP13, CDS26, NBA14+}. The literature for semi-implicit schemes for the Euler equations is quite abundant \cite{AS20, DT11, BQR19, SAM24, NBA14+, DLV17}.
However, mere satisfaction of the AP property will not enable the scheme to perform well in the weakly asymptotic regime because of the gravitational source term.
As most of the flows under study in the low Mach regimes are perturbations of a steady state, it is absolutely critical for a numerical scheme to not amplify small perturbations to steady states.
This is usually achieved by obtaining discretizations which exactly preserve certain steady states, i.e., are well-balanced (WB) \cite{Chandrashekar2015, BLY17}. For example, the temporal semi-discretization
\begin{align}
    \frac{\rho^{n+1} - \rho^n}{\Dlt} + \dvg \vq^{n+1} = 0, \\
    \frac{\vq^{n+1} - \vq^{n}}{\Dlt} + \dvg \left(\frac{\vq^n \otimes \vq^n}{\rho^n} \right) + \frac{\grd p(\rho^{n+1})}{\veps^2} = -\frac{\rho^{n+1} \grd \phi}{\veps^2},
\end{align}
presented in \cite{AS20_AIMS} is formally asymptotic preserving.
\begin{figure}[htb]
\vspace{-3mm}
  \centering
  \includegraphics[height=0.2\textheight]{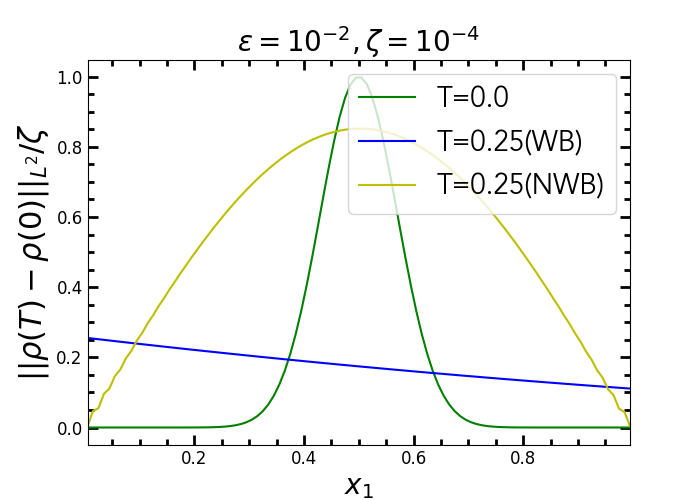} \qquad
\includegraphics[height=0.2\textheight]{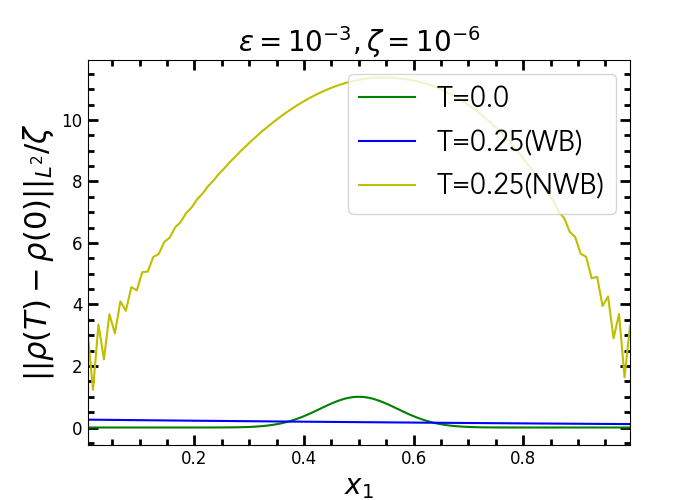}
  \vspace{-3mm}
    \caption{Comparison between a non-well-balanced AP and a well-balanced AP scheme with $\veps = 10^{-2}$ (left) and $\veps = 10^{-3}$ (right) for a perturbation with amplitude $\zeta = \veps^2$ of a steady state \cite{LB99, AKS22}.}
  \label{fig:WB_vs_NWB}
\end{figure}

Figure~\ref{fig:WB_vs_NWB} compareas the above AP but non-well-balanced (NWB) scheme and the AP-WB scheme developed in this work.
Initially, we add a small perturbation of amplitude $\zeta = \veps^2$ to the steady state equilibrium \eqref{eq:rho_eq} (see Section~\ref{ssec:hydro_stat}, Case~2, for details).
For $\veps = 10^{-2}$, though both schemes damp the initial perturbation, the WB scheme clearly outperforms the NWB scheme.
However, the NWB scheme gives abysmal results by amplifying the initial perturbation for $\veps = 10^{-3}$.
In contrast, the WB scheme still damps the initial perturbation as expected.
This showcases the need for well-balancing while deriving a spatial discretization.

In \cite{TPK20} the authors have presented an AP-WB scheme for Euler equations with gravity and an energy equation using the Suliciu relaxation method.
In \cite{BLY17} the authors have developed AP-WB schemes using IMEX methods for the perturbations of variables.
We tread a different path to design an AP-WB scheme: we combine a penalization approach \cite{DP13}, hydrostatic reconstruction of steady-state variables, and prudent source discretization \cite{Chandrashekar2015}.
The penalization relies on the derivation of a linearized balance against the non-linear balance equation \eqref{eq:hydro_bal_an}, but essentially having the same solution as \eqref{eq:rho_eq}.
This allows us to be linearly implicit, hence avoiding expensive Newton solvers.
The source pressure balance obtained following \cite{Chandrashekar2015} and a novel discretization of the linear balance yields an AP-WB first-order scheme in time.
Higher orders in time are achieved by using the additive IMEX framework \cite{PR01}.

The rest of the paper is organized as follows. Section~\ref{sec:time-discrete} is dedicated to the development and AP analysis of the first-order in time semi-discrete scheme, including the derivation of the linear balance and the penalized momentum equation.
In Section~\ref{sec:ho_tsd}, higher-order extensions are obtained using additive IMEX Runge-Kutta (RK) schemes.
Spatial discretization of the time-semi-discrete scheme, in particular, a hydrostatic reconstruction and appropriate source-term discretization combined with a second-order finite volume (FV) discretization, thus yielding an implementable fully-discrete scheme, is presented in Section~\ref{sec:fully-disc}.
Moreover, it is proved that the fully discrete scheme is well-balanced. In Section~\ref{sec:numerical_results}, results of numerical case studies are presented to corroborate the AP and WB property of the designed penalized IMEX-RK schemes.
Finally, we close the paper in Section~\ref{sec:conclusion} with some conclusions and future prospects.
\section{First-order time discretization and its analysis}
\label{sec:time-discrete}
To be asymptotic preserving in the limit of the Mach number going to zero, i.e., $\veps \to 0$, we follow the seminal works of \cite{DT11,ap_jin, HJL12, DP13, DLV17, BQR19, BLY17}. The idea is to introduce some amount of implicitness to the numerical scheme judiciously so that the resulting scheme is able to transition to a consistent discretization of the limit system. To achieve this, the requirement is to identify terms (specifically flux terms) which are to be treated implicitly. The problem is that all the flux terms in the momentum equation are non-linear and, therefore, will pose challenges in the form of iterative Newton solvers when treated implicitly. The next design feature is that the first-order scheme must allow a straightforward higher-order extension using the classical additive IMEX-RK framework.

This section is organized as follows. We begin with some necessary definitions, pivotal to the analysis of the schemes to be designed.
Following that we recognize a few terms that are introduced to the momentum balance equation to aid in developing the time discretization (without the need of Newton solvers); similar developments have been reported in \cite{DP13} for kinetic systems and for fluid systems in \cite{CDS26, SAM24}.
This leads to the presentation of the design of the first-order time semi-discrete scheme.
Finally, the section ends with AP analysis of the developed scheme.
\subsection{Well-prepared and ill-prepared initial data}
\begin{definition}[cf.~\cite{FKM19}]\label{def:wellprep_id}
    A tuple $(\rho, \vu)$ of solutions to the compressible Euler equations \eqref{eq:iseuler_mass_ane}--\eqref{eq:iseuler_mom_ane} is said to be well-prepared if the following decompositions hold:
    \begin{align}
        \rho &= \rho_{(0)} + \veps^2 \rho_{(2)} \label{eq:rho_md}\\
        \vu &= \vu_{(0)} + \veps \vu_{(1)} \label{eq:u_md}
    \end{align}
    where:
    \begin{equation}\label{eq:wp_inc}
        \rho_{(0)} = \rho_\mathrm{eq} \qquad \mbox{and} \qquad \dvg(\rho_{(0)} \vu_{(0)}) = 0.
    \end{equation}
\end{definition}
\begin{remark}
It has to be noted that although well-preparedness of the initial data is a sufficient condition for the solution of the compressible system to converge to the solution of the incompressible counterpart, it is not absolutely necessary.
\end{remark}
\begin{definition}[cf.~\cite{FKM19}]\label{def:illprep_id}
    A tuple $(\rho, \vu)$ of solutions to the compressible Euler equations \eqref{eq:iseuler_mass_ane}--\eqref{eq:iseuler_mom_ane} is said to be ill-prepared if the decompositions \eqref{eq:rho_md}--\eqref{eq:u_md} hold, but we only have that
    \begin{equation*}
        ||\rho_{(0)}||_{L^2}  \quad \mbox{and} \quad ||\dvg(\rho_{(0)} u_{(0)})||_{L^2} \quad \mbox{are bounded}
    \end{equation*}
    instead of \eqref{eq:wp_inc}.
\end{definition}
\begin{remark}
    In \cite{FKM19} the authors have shown that in spite of the ill-preparedness of the initial data for the isentropic Euler equation without gravitational source term, still the compressible solution at least locally converges to its incompressible counterpart in the limit of $\veps \to 0$. Therefore, it would be highly desirable for a numerical scheme to perform a similar transitory act.
\end{remark}
\subsection{Equilibrium reformulation}
The gradient of the equilibrium solution \eqref{eq:rho_eq} is
\begin{align} \label{eq:hydro_bal_lin}
    \nabla \rho_\mathrm{eq} = -\frac{1}{\gamma} \rho_\mathrm{eq}^{(2-\gamma)} \grd \phi.
\end{align}
Note that when $\veps \to 0$,  $ \rho \to \rho_\mathrm{eq}$. Using this and \eqref{eq:hydro_bal_lin}, we get
\begin{align*}
& \nabla \rho + \frac{1}{\gamma} \rho^{(2-\gamma)} \grd \phi \xrightarrow{\hspace{3mm}\veps \to 0 \hspace{3mm}} \nabla \rho_\mathrm{eq} + \frac{1}{\gamma} \rho^{(2-\gamma)}_\mathrm{eq} \grd \phi = 0.
\end{align*}
%\[ \nabla \rho + \frac{e_3}{1 -\gamma} \rho^{(1-\gamma)} = \nabla \rho + \frac{e_3}{1 -\gamma} \frac{\rho}{\rho^\gamma}\]
Assuming that $\rho_\mathrm{eq} > 0 $, which is a fair assumption since it is the density as $\veps \to 0$, we obtain
\begin{align*}
& \nabla \rho + \frac{1}{\gamma} \frac{\rho}{\rho_\mathrm{eq}^{(\gamma - 1)}} \nabla \phi   \xrightarrow{\hspace{3mm}\veps \to 0 \hspace{3mm}} \nabla \rho_\mathrm{eq} + \frac{1}{\gamma} \rho_\mathrm{eq}^{(2-\gamma)}\nabla \phi = 0.
\end{align*}
The following equality can be obtained and used in the above expressions to simplify them:
\[\frac{\nabla \phi}{\gamma \rho_\mathrm{eq}^{(\gamma - 1)}} = -\frac{\nabla\rho_\mathrm{eq}}{\rho_\mathrm{eq}}. \]
Note that $\rho_\mathrm{eq}$ is the unique solution to the non-linear problem \eqref{eq:hydro_bal_an} as well as for the linear problem
\begin{equation} \label{eq:lin_hydrobal}
    \nabla \rho - \frac{\rho}{\rho_\mathrm{eq}} \nabla \rho_\mathrm{eq} = 0.
\end{equation}
This linear balance \eqref{eq:lin_hydrobal} which has the same solution $\rho_\mathrm{eq}$ as the non-linear balance \eqref{eq:hydro_bal_an} up to a constant multiple  $K$ (to be fixed to 1 by boundary conditions) allows us to maintain the non-linear balance by just maintaining the linear one.
The idea is to utilize the linear hydrostatic balance \eqref{eq:lin_hydrobal} to obtain a stiff-non-stiff splitting leading to an IMEX scheme.
To this end, we follow \cite{SAM24} and reformulate the momentum balance equation \eqref{eq:iseuler_mom_ane} by adding and subtracting the balanced terms in \eqref{eq:lin_hydrobal} to obtain the momentum equation
\begin{equation}
    \Dt (\rho \vu ) + \grd \cdot (\rho \vu \otimes \vu ) + \frac{\grd p}
  {\veps^2}  + \frac{\rho \nabla \phi}{\veps^2} -  \frac{1}{\veps^2} \left(\nabla \rho - \frac{\rho}{\rho_\mathrm{eq}} \nabla \rho_\mathrm{eq} \right) + \frac{1}{\veps^2} \left(\nabla \rho - \frac{\rho}{\rho_\mathrm{eq}} \nabla \rho_\mathrm{eq}  \right) = 0.
\end{equation}
\begin{remark}
    The addition and subtraction of terms has two fundamental aims.
    The first and obvious one is to get linear terms associated to the non-linear balance \eqref{eq:hydro_bal_an} (to provide a choice of linear terms for implicit discretization).
    Secondly, discretizing the newly introduced cancelling terms at two different time levels (one explicit and the other implicit) introduces a numerical error. This penalizing error contributes toward stabilizing the scheme to tackle smaller values of Mach numbers. This in turn leads to the asymptotic preserving property for the time discretization.
\end{remark}
\subsection{First-Order Time Discretization}
The first-order accurate semi-implicit time discretization is obtained by treating the mass flux fully implicitly and making a judicious choice of splitting for the momentum fluxes and source terms.
\begin{definition}\label{def:1st_TSD}
    The first-order time semi-discrete scheme for the Euler equations \eqref{eq:iseuler_mass_ane}--\eqref{eq:iseuler_mom_ane} updates the solution $(\rho^n, \vq^n)$ at time $t^n$ to $(\rho^{n + 1}, \vq^{n+1})$ at time $t^{n+1} = t^n + \Delta t$ semi-implicitly as
    \begin{align}
         \frac{\rho^{n+1} - \rho^{n}}{\Dlt} + & \dvg \vq^{n+1} = 0, \label{eq:mass_fo}\\
         \frac{\vq^{n+1} - \vq^{n}}{\Dlt} + &\dvg \left(\frac{\vq^n \otimes \vq^n}{\rho^n}\right) + \frac{1}{\veps^2} \left( \grd p^n + \rho^n \nabla \phi \right) - \frac{1}{\veps^2} \left(\nabla \rho^n -\frac{\rho^n}{\rho_\mathrm{eq}} \nabla \rho_\mathrm{eq}\right) + \notag \\
         &  \frac{1}{\veps^2} \left(\nabla \rho^{n+1} - \frac{\rho^{n+1}}{\rho_\mathrm{eq}} \nabla \rho_\mathrm{eq} \right) = 0.\label{eq:mom_fo}
    \end{align}
\end{definition}
Following the approach of \cite{DT11} we obtain a reformulation of the above scheme. The numerical solution computed is the solution of this reformulation rather than the scheme in Definition~\ref{def:1st_TSD}. To this end, the momentum update \eqref{eq:mom_fo} can be written as
\begin{equation}\label{eq:mom_fo_wh}
    \vq^{n+1} = \hat{\vq}^{n+1} - \frac{\Dlt}{\veps^2} \left(\nabla \rho^{n+1} -\frac{\rho^{n+1}}{\rho_\mathrm{eq}} \nabla \rho_\mathrm{eq} \right),
\end{equation}
where $\hat{\vq}^{n+1}$ is the part of $\vq^{n+1}$ that can be updated explicitly, i.e.,
\begin{equation}\label{eq:qh_def}
    \hat{\vq}^{n+1} :=\vq^{n} -\Dlt \dvg \left(\frac{\vq^n \otimes \vq^n}{\rho^n}\right) - \frac{\Dlt}{\veps^2} \left( \grd p^n + \rho^n \nabla \phi \right) + \frac{\Dlt}{\veps^2} \left( \nabla \rho^n -\frac{\rho^n}{\rho_\mathrm{eq}} \nabla \rho_\mathrm{eq} \right).
\end{equation}
Substituting the expression for $\vq^{n+1}$ from \eqref{eq:mom_fo_wh} in the mass update \eqref{eq:mass_fo} we get the following elliptic problem for $\rho^{n+1}$:
\begin{equation}
    \rho^{n+1} = \rho^n - \Dlt \dvg \left(\hat{\vq}^{n+1} - \frac{\Dlt}{\veps^2} \left(\nabla \rho^{n+1} -\frac{\rho^{n+1}}{\rho_\mathrm{eq}} \nabla \rho_\mathrm{eq} \right) \right).
\end{equation}
\begin{definition} \label{def:1st_TSD_rf}
    The first-order reformulated time semi-discrete scheme for the Euler equations \eqref{eq:iseuler_mass_ane}--\eqref{eq:iseuler_mom_ane} updates the solution $(\rho^n, \vq^n)$ at time $t^n$ to $(\rho^{n + 1}, \vq^{n+1})$ at time $t^{n+1}$ by first evaluating $\hat{\vq}^{n+1}$ via \eqref{eq:qh_def}, followed by solving the ensuing linear elliptic problem, to update $\rho^{n+1}$,
    \begin{equation}\label{eq:elliptic_fo}
        -\frac{\Dlt^2}{\veps^2} \Delta \rho^{n+1} + \frac{\Dlt^2}{\veps^2} \dvg  \left(\frac{\rho^{n+1}}{\rho_\mathrm{eq}} \nabla \rho_\mathrm{eq} \right) + \rho^{n+1} = \rho^n - \Dlt \dvg \hat{\vq}^{n+1}.
    \end{equation}
    Once $\rho^{n+1}$ is obtained, the momentum $\vq^{n+1}$ is computed explicitly via \eqref{eq:mom_fo_wh}.
\end{definition}
\begin{algorithm}
\caption{First-Order Scheme: \quad $(\rho^n, \vq^n) \to (\rho^{n+1}, \vq^{n+1})$}
\begin{algorithmic}[1]
\STATE Compute $\hat{\vq}^{n+1}$ explicitly using \eqref{eq:qh_def}
\STATE Solve the elliptic problem \eqref{eq:elliptic_fo} for $\rho^{n+1}$
\STATE Update $\vq^{n+1}$ explicitly using \eqref{eq:mom_fo_wh}
\end{algorithmic}
\end{algorithm}
\begin{proposition}
    The first-order time semi-discrete scheme in Definition~\ref{def:1st_TSD} and the reformulated one in Definition~\ref{def:1st_TSD_rf} are equivalent.
\end{proposition}
\begin{remark}
    This reformulation technique was introduced in \cite{DT11}. It renders an equivalent scheme which is at least twice (depending on spatial dimensions) as efficient as the original semi-implicit scheme since only one matrix inversion is required to solve the elliptic problem for the density $\rho^{n+1}$, avoiding the need to take $\vq^{n+1}$ as an unknown for the implicit scheme and thus leading to a much smaller system matrix.
\end{remark}
\subsection{Asymptotic Consistency}
The AP property \cite{Jin12} of a numerical scheme is linked to two facets associated with any numerical discretization. The first is consistency: the ability to yield a consistent discretization of the limit system in the limit of the singular perturbation parameter going to zero, called the asymptotic consistency. The second one is stability: the ability of the scheme to remain stable independent of the singular perturbation parameter, called asymptotic stability. Here we study the consistency of the first-order time semi-discrete scheme for $\veps \to 0$.
\begin{proposition} \label{prop:ap_1st}
    Suppose that $(\rho^n, \vu^n)$ at time $t^n$ are well-prepared, i.e., they admit the decomposition in Definition~\ref{def:wellprep_id}. Then, if the solution $(\rho^{n + 1}, \vu^{n + 1})$ defined by the scheme in Definition~\ref{def:1st_TSD} admits the multi-scale decomposition \eqref{eq:rho_md}--\eqref{eq:u_md}, it is well-prepared as well. Further, the time semi-discrete scheme in Definition~\ref{def:1st_TSD} transforms to a consistent discretization of the incompressible limit system, in other words, the first-order time semi-discretization is weakly asymptotically consistent.
\end{proposition}
\begin{proof}
    We start by stating the assumption of the multi-scale decomposition for $(\rho^{n+1}, \vu^{n+1})$, i.e.,
    \begin{equation}
        \rho^{n+1} = \rho_{(0)}^{n+1} + \veps^2 \rho_{(2)}^{n+1}, \qquad
        \vu^{n+1} = \vu_{(0)}^{n+1} + \veps \vu_{(1)}^{n+1}.
    \end{equation}
Plugging this into the momentum update \eqref{eq:mom_fo}, the $\mco (\veps^{-2})$ terms yield
\begin{equation}\label{eq:oepsm2_den}
    \left( \nabla p_{(0)}^n + \rho_{(0)}^n \nabla \phi \right)
    - \left( \nabla \rho^{n}_{(0)} -\frac{\rho^n_{(0)}}{\rho_\mathrm{eq}} \nabla \rho_\mathrm{eq} \right)
    + \left( \nabla \rho^{n+1}_{(0)} -\frac{\rho^{n+1}_{(0)}}{\rho_\mathrm{eq}} \nabla \rho_\mathrm{eq} \right)
    = 0.
\end{equation}
Since $\rho_{(0)}^n$ is well-prepared, i.e., $\rho_{(0)}^n$ satisfies \eqref{eq:hydro_bal_an}, the expression in the first bracket above is zero.
Similarly, the terms in the second bracket also vanish since $\rho^n$ is well-prepared, i.e., $\rho_{(0)}^n = \rho_\mathrm{eq}$ and by the equality in \eqref{eq:hydro_bal_lin}, the linearized hydrostatic balance.
Therefore, the above equation boils down to
\begin{equation*}
\nabla \rho^{n+1}_{(0)} -\frac{\rho^{n+1}_{(0)}}{\rho_\mathrm{eq}} \nabla \rho_\mathrm{eq} = 0,
\end{equation*}
which can be solved for $\rho^{n+1}_{(0)}$ uniquely up to a constant $K$ by imposing appropriate (periodic or no-flux) boundary conditions to obtain
\begin{equation}
    \rho^{n+1}_{(0)} = K \rho_\mathrm{eq}.
\end{equation}
Proceeding further by considering the $\mco (\veps^0)$ terms in the mass update \eqref{eq:mass_fo} we get
\begin{equation}\label{eq:mass_fo_o1}
    \frac{\rho_{(0)}^{n+1} - \rho_{(0)}^{n}}{\Dlt} = -\dvg (\rho_{(0)}^{n+1} \vu^{n+1}_{(0)}).
\end{equation}
Integrating over a domain $\Omega$, using Gauss' divergence theorem and imposing periodic or no-flux boundary conditions, the right-hand-side integral vanishes, yielding
\[\frac{1}{\Dlt}\int_{\Omega} K\rho_\mathrm{eq} -\rho_\mathrm{eq} = 0.\]
Thus, $K = 1$ or $\int_{\Omega} \rho_\mathrm{eq} = 0$; since the second is in general not true we conclude that $K = 1$, which gives that the density $\rho^{n+1}$ is well-prepared i.e. $\rho^{n+1}_{(0)} = \rho_\mathrm{eq}$.

Now using the above in the $\mco (\veps^0)$ terms of the mass update \eqref{eq:mass_fo_o1} we obtain the asymptotic divergence condition \eqref{eq:ane_lim_div} for time $t^{n+1}$ as
\begin{equation}
    \dvg (\rho_{(0)}^{n+1} \vu_{(0)}^{n+1}) = 0.
\end{equation}
Finally, considering $\mco(\veps^0)$ terms of the momentum update \eqref{eq:mom_fo} gives
\begin{equation}
\begin{aligned}
\frac{\vq^{n+1}_{(0)} - \vq^{n}_{(0)}}{\Dlt} + \dvg (\rho_{(0)}^n \vu_{(0)}^n \otimes \vu^n_{(0)}) \qquad &
\\
+ (\nabla p_{(2)}^n  + \rho_{(2)}^n \nabla \phi) - \nabla (\rho_{(2)}^{n} - \rho_{(2)}^{n+1}) &+ \frac{(\rho_{(2)}^n - \rho_{(2)}^{n+1})}{\rho_\mathrm{eq}}\nabla \rho_\mathrm{eq} =0.
\end{aligned}
\end{equation}
This is a consistent discretization of the limit equation \eqref{eq:ane_lim_mom}, where $\nabla p_{(2)}^n$ is the incompressible pressure term and $\rho_{(2)}^n \nabla \phi$ is the source term for the incompressible limit system. The extra terms, i.e., $\nabla (\rho_{(2)}^{n} - \rho_{(2)}^{n+1})$ and $ \frac{\rho_{(2)}^n - \rho_{(2)}^{n+1}}{\rho_\mathrm{eq}} \nabla \rho_\mathrm{eq}$ are, respectively, first-order discretizations of $\nabla(\rho_{(2)} - \rho_{(2)})$ and $\frac{(\rho_{(2)} - \rho_{(2)})}{\rho_\mathrm{eq}} \nabla \rho_\mathrm{eq}$, which are essentially zero terms (up to the order of the time integration scheme).
\end{proof}
\begin{remark}
    The above proof for asymptotic consistency requires well-prepared data $(\rho^n, \vu^n)$ at time $t^{n}$. This is called weak asymptotic consistency \cite{SAM24}.
\end{remark}
    Suppose the initial condition is not well-prepared (Definition~\ref{def:wellprep_id}).
    Now, considering $\mco(\veps^{-2})$ terms in \eqref{eq:mom_fo} yields \eqref{eq:oepsm2_den}.
    The lack of well-preparedness implies
    \[\nabla p_{(0)}^n + \rho_{(0)}^n \nabla\phi \neq 0, \qquad  \nabla \rho^n_{(0)} - \frac{\rho^n_{(0)}}{ \rho_\mathrm{eq}} \nabla \rho_\mathrm{eq} \neq 0.\]
    As a consequence, we do not have the well-preparedness for $\rho^{n+1}$.
    The inability of the first-order scheme to furnish a well-prepared density hinders the extraction of the divergence-free condition for the limiting velocity.
    Hence, the first-order time semi-discrete scheme in Definition~\ref{def:1st_TSD} is not strongly asymptotic preserving.
\begin{remark}
    For strongly asymptotic consistency, type-A IMEX-RK schemes \cite{DP13, KENNEDY2003139} have to be considered, as discussed in the next section.
\end{remark}
\section{Higher-order time discretization}
\label{sec:ho_tsd}
Next, we generalize the scheme to obtain higher-order time discretizations using the mathematical framework discussed in \cite{PR01, AscherRuuthSpiteri}.
This class of additively-split IMEX-RK schemes provides an uncomplicated, robust, and efficient base to design schemes that are uniformly stable with respect to fast-scale dynamics.
In the recent past, they have been extensively used  \cite{AS20,DP13,ap_jin,Jin12,PR01,SAM24} to obtain higher-order methods consistent with the limit of singular perturbation problems.
\subsection{Additively split IMEX-RK methods}
We restrict ourselves to a specific subclass of IMEX-RK schemes; namely, we rely on diagonally implicit (DIRK) schemes. This is a typical choice in many different contexts where computational cost is one of the main concerns. In general, an $s$-stage IMEX-RK scheme is characterized by the two $s\times s$ lower triangular matrices $\tilde{A} = (\tilde{a}_{i,j})$, and $A=(a_{i,j})$, the coefficients $\tilde{c}=(\tilde{c}_1,\tilde{c}_2,\ldots,\tilde{c}_s)$ and $c=(c_1, c_2, \ldots, c_s)$, and the weights $\tilde{\omega}=(\tilde{\omega}_1, \tilde{\omega}_2,
\ldots,\tilde{\omega}_s)$ and $\omega = (\omega_1, \omega_2, \ldots, \omega_s)$.
Here, for the DIRK class, the entries of $\tilde{A}$ and $A$ satisfy the conditions $\tilde{a}_{i,j}=0$ for $ j \geq i$, and $a_{i,j}=0$ for $j>i$. To introduce the method, let us consider the following general additively split stiff system of ODEs:
\begin{equation}
  \label{eq:stiff_ODE}
  y^\prime = f(t,y) + \frac{1}{\veps^2} g(t,y),
\end{equation}
where $0<\veps\ll 1$ is called the stiffness parameter. The functions $f$ and $g$ are known as, respectively, the non-stiff part and the stiff part of the system \eqref{eq:stiff_ODE} (see, e.g.\ \cite{HW96, Bos07}, for a comprehensive treatment of such systems). Let $y^n$ be a solution of \eqref{eq:stiff_ODE} at time $t^n$ and let $\Dlt$ denote a fixed time step. An $s$-stage IMEX-RK scheme updates $y^n$ to $y^{n+1}$ through the computation of $s$ intermediate stage updates:
\begin{equation}
  Y^{(i)} = y^n + \Dlt \sum\limits_{j=1}^{i-1}\tilde{a}_{i,j} f(t^n + \tilde{c}_j\Dlt, Y^{(j)}) +
        \Dlt\sum \limits_{j=1}^i a_{i,j} \frac{1}{\veps^2} g(t^n + c_j
        \Dlt,Y^{(j)}), \ 1 \leq i \leq s, \label{eq:imex_Yi}
\end{equation}
followed by the final update:
\begin{equation}
     y^{n+1} = y^n  + \Dlt \sum\limits_{i=1}^{s}\tilde{\omega}_{i} f(t^n
            + \tilde{c}_i\Dlt, Y^{(i)}) + \Dlt\sum \limits_{i=1}^s
            \omega_{i}\frac{1}{\veps^2} g(t^n +
            c_i\Dlt,Y^{(i)}), \label{eq:imex_yn+1}
\end{equation}
Typically IMEX-RK schemes are represented by a pair of Butcher tableaux
\begin{equation}
\label{eq:butcher_tableau}
    \begin{array}{c|c}
       \tilde{c} & \tilde{A} \\\hline
         & \tilde{\omega}^T
    \end{array}
    \qquad
    \begin{array}{c|c}
       c & A \\\hline
         & \omega^T
    \end{array}
\end{equation}
The values of the RK coefficients of the Butcher tableaux in \eqref{eq:butcher_tableau} for a particular IMEX-RK scheme satisfy the particular order conditions; see \cite{PR01}.

We consider the following two subcategories of IMEX-RK schemes, namely type-A and type-CK, which are defined below following \cite{boscarino, DP13, KENNEDY2003139}.
\begin{definition}
\label{eq:typeA_CK}
An IMEX-RK scheme is said to be of
\begin{itemize}
\item type-A, if the matrix $A$ is invertible;
\item type-CK, if the matrix $A \in \mbb{R}^{s \times s}, \ s \geq 2$,
  can be written as
$
  A =
  \begin{pmatrix}
    0 & 0 \\
    \alpha & A_{s-1 }
  \end{pmatrix},
$
where $\alpha \in \mbb{R}^{s-1} $ and $A_{s-1} \in \mbb{R}^{s-1 \times
s-1}$ is invertible.
\end{itemize}
\end{definition}
Another property which is crucial to obtain the AP property for non-linear problems is that the IMEX-RK scheme is globally stiffly accurate \cite{BR13}.
\begin{definition}
\label{eq:GSA}
An IMEX-RK scheme with the Butcher tableaux given in
\eqref{eq:butcher_tableau} is said to be globally stiffly
accurate (GSA), if
\begin{equation}
  \tilde{a}_{s,j} = \tilde{\omega}_{j}
  \quad \text{and} \quad
  a_{s,j} = \omega_{j}
  \qquad
  \text{for all } j = 1,
\ldots, s.
\end{equation}
\end{definition}
This property is enough to guarantee that the solution obtained from the last stage of the IMEX-RK method coincides with the numerical solution $y^{n+1}=Y^{(s)}$.
\subsection{Higher-order time discretization for the Euler equations with gravity}
To obtain higher-order schemes we consider all the terms in the first-order discretization with a superscript $n$ as the explicit/non-stiff part and the ones treated implicitly, denoted with a superscript of $n+1$, as the stiff part.
\begin{definition}\label{def:ho_TSD}
    The s-stage IMEX-RK scheme for the Euler equations \eqref{eq:iseuler_mass_ane}--\eqref{eq:iseuler_mom_ane} updates the solution $(\rho^n, \vq^n)$ at time $t^n$ to $(\rho^{n + 1}, \vq^{n+1})$ at time $t^{n+1} = t^n + \Delta t$ via  the following $s$ intermediate stages:
\begin{eqnarray}
\rho^{k} &=& \rho^{n} - \Dlt  \sum_{\ell = 1}^k a_{k,\ell} \dvg \vq^{\ell}, \label{eq:mass_ho_k} \\
\vq^{k} &=& \vq^{n} - \Dlt \sum_{\ell = 1}^{k-1} \tilde{a}_{k, \ell} \left(\dvg \left(\frac{\vq^{\ell} \otimes \vq^{\ell}}{\rho^{\ell}}\right) + \frac{1}{\veps^2} \left( \grd p^{\ell} + \rho^{\ell} \nabla \phi \right) \right.\nonumber\\
&&\left.- \frac{1}{\veps^2}\left(\nabla \rho^{\ell} - \frac{\rho^{\ell}}{\rho_\mathrm{eq}} \nabla \rho_\mathrm{eq}\right) \right) - \Dlt \sum_{\ell = 1}^{k} a_{k, \ell} \frac{1}{\veps^2} \left(\nabla \rho^{\ell} - \frac{\rho^{\ell}}{\rho_\mathrm{eq}} \nabla \rho_\mathrm{eq} \right). \label{eq:mom_ho_k}
\end{eqnarray}
where $k = 1,\ldots,s$. As we choose only GSA schemes the final update is the same as the update at the last $(s^{th})$ stage, i.e., the numerical solution at time $t^{n+1}$ is given by
\begin{equation}
    (\rho^{n+1}, \vq^{n+1}) = (\rho^{s}, \vq^{s}).
\end{equation}
\end{definition}
Analogous to the first-order scheme, we obtain an equivalent semi-discretization using an elliptic reformulation by writing the mass \eqref{eq:mass_ho_k} and momentum \eqref{eq:mom_ho_k} updates as
\begin{align}
    \rho^k &= \hat{\rho}^k - \Dlt a_{k,k} \dvg \vq^k, \label{eq:rho_k_ho_rec}\\
    \vq^k &= \hat{\vq}^k - \frac{ \Dlt a_{k,k}}{\veps^2} \left( \nabla \rho^{k} - \frac{\rho^k}{\rho_\mathrm{eq}} \nabla \rho_\mathrm{eq} \right), \label{eq:q_k_ho_rec}
\end{align}
where $\hat{\rho}^k$ and $\hat{\vq}^k$ are, respectively, the explicit mass and momentum updates, i.e.,
\begin{eqnarray}
\hat{\rho}^{k} &=& \rho^{n} - \Dlt  \sum_{\ell = 1}^{k-1} a_{k,\ell} \dvg \vq^{\ell}, \label{eq:mass_ho_khat} \\
\hat{\vq}^k &=& \vq^{n} - \Dlt \sum_{\ell = 1}^{k-1} \tilde{a}_{k, \ell} \left( \dvg \left(\frac{\vq^{\ell} \otimes \vq^{\ell}}{\rho^{\ell}}\right) + \frac{1}{\veps^2} \left( \grd p^{\ell} + \rho^{\ell} \nabla \phi \right) \right.\nonumber\\
&&\left.- \frac{1}{\veps^2}\left(\nabla \rho^{\ell} - \frac{\rho^{\ell}}{\rho_\mathrm{eq}} \nabla \rho_\mathrm{eq}\right) \right) - \Dlt \sum_{\ell = 1}^{k-1} a_{k, \ell} \frac{1}{\veps^2} \left(\nabla \rho^{\ell} - \frac{\rho^{\ell}}{\rho_\mathrm{eq}} \nabla \rho_\mathrm{eq} \right). \label{eq:mom_ho_khat}
\end{eqnarray}
Using \eqref{eq:q_k_ho_rec} in the mass equations \eqref{eq:rho_k_ho_rec} we obtain a linear elliptic problem for $\rho^{k}$:
\begin{equation}\label{eq:elliptic_ho}
    -\frac{\Dlt^2 a_{k,k}^2}{\veps^2} \Delta \rho^{k} +\frac{\Dlt^2 a_{k,k}^2}{\veps^2} \dvg  \left(\frac{\rho^{k}}{\rho_\mathrm{eq}} \nabla \rho_\mathrm{eq} \right) + \rho^{k} = \hat{\rho}^k - \Dlt a_{k,k} \dvg \hat{\vq}^{k}.
\end{equation}
\begin{definition}\label{def:ho_TSD_ref}
    The $s$-stage reformulated IMEX-RK time semi-discrete scheme for the Euler equations \eqref{eq:iseuler_mass_ane}--\eqref{eq:iseuler_mom_ane} updates the solution $(\rho^n, \vq^n)$ at time $t^n$ via $s$ intermediate stages, to $(\rho^{n+1}, \vq^{n+1})$ at time $t^{n+1}$.
    In stage $k$, $\hat{\rho}^k$ and $\hat{\vq}^k$ are obtained explicitly via \eqref{eq:mass_ho_khat} and \eqref{eq:mom_ho_khat}, respectively, followed by the solution of the elliptic problem \eqref{eq:elliptic_ho} for $\rho^k$. Finally, $\vq^k$ is updated explicitly via \eqref{eq:q_k_ho_rec}. The numerical solution $(\rho^{n+1}, \vq^{n+1})$ is the solution of the last $s^{th}$ stage (obtained by GSA imposition).
\end{definition}
\begin{proposition}
    The higher-order time semi-discrete scheme in Definition~\ref{def:ho_TSD} and the reformulated one in Definition~\ref{def:ho_TSD_ref} are equivalent.
\end{proposition}
\begin{proposition}
    The higher-order time semi-discrete scheme in Definition~\ref{def:ho_TSD} for a DIRK scheme is weakly asymptotically consistent with the asymptotic limit system in the limit of $\veps \to 0$.
\end{proposition}
\begin{remark}
    The proof traces similar arguments as for the first-order scheme and employs induction along the lines of the proofs for asymptotic consistency in \cite{SAM24}. Note that the first-order IMEX-RK scheme can be deduced from the higher-order scheme in Definition~\ref{def:ho_TSD} by using the ARS(1,1,1) IMEX-RK scheme in \eqref{eq:ars_dp}.
\end{remark}
\begin{equation}\label{eq:ars_dp}
\begin{aligned}
    \begin{array}{c|c c}
    0 & 0 & 0 \\
    1 & 1 & 0\\
    \hline
      & 1	& 0
  \end{array}
  \hspace{10pt}
  \begin{array}{c|c c}
    0 & 0 & 0 \\
    1 & 0 & 1 \\
    \hline
        & 0 & 1
  \end{array}
  \qquad
  &
  \qquad
  \begin{array}{c|c c}
    0 & 0 & 0 \\
    1 & 1 & 0\\
    \hline
      & 1	& 0
  \end{array}
  \hspace{10pt}
  \begin{array}{c|c c}
    \beta & \beta	& 0 \\
    1	& 1 - \beta & \beta \\
    \hline
        & 1 - \beta & \beta
  \end{array} \\
  \mbox{ARS(1,1,1), type-CK, GSA  \cite{AscherRuuthSpiteri}} \qquad &  \mbox{DP-A(1,2,1), type-A, GSA   $\beta > 1/2$ \cite{DP13}}
  \end{aligned}
\end{equation}
\vspace{-3mm}
\begin{remark}
    A detailed linear stability analysis of the above scheme for the homogeneous problem, i.e., without gravity is carried out in \cite{AS20}. It has been shown that the higher-order time discretization is linearly $L^2$ stable under a hyperbolic CFL condition. The stability constraint is independent of the Mach number $\veps$ and only depends on the IMEX-RK coefficients.
\end{remark}
\subsection{Strong asymptotic consistency}
As stated in Section~\ref{sec:time-discrete}, the first-order scheme relies on well-prepared initial data to be asymptotically consistent. This deficiency or rather strict requirement on the initial data can be relaxed for type-A  schemes. To this end we show that for type-A schemes (e.g.  DP-A(1,2,1) in \eqref{eq:ars_dp}), asymptotic consistency of the general IMEX-RK scheme can be obtained after two time steps, even for ill-prepared initial data (Definition~\ref{def:illprep_id}).
\begin{proposition}
    Suppose $(\rho^n, \vu^n)$, the datum at time $t^n$ is ill-prepared (Definition~\ref{def:illprep_id}). Let the solution $(\rho^{n+1}, \vu^{n+1})$ at time $t^{n+1}$ admit the multiscale decomposition \eqref{eq:ansatz} for the $s$-stage
    IMEX-RK scheme in Definition~\ref{def:ho_TSD} with type-A tableaux. Then, the solution after two time steps, i.e., $(\rho^{n+2}, \vu^{n+2})$ is well-prepared and the scheme reduces to a consistent discretization of the anelastic limit system \eqref{eq:ane_lim_div}--\eqref{eq:ane_lim_mom}.
\end{proposition}
\begin{proof}
We divide the proof into two parts: first the iteration $n \to n+1$ and next the update $n+1 \to n+2$. We then apply induction on the number of intermediate stages $k$ to each iteration.
\\
\underline{\textbf{First time-step $t^n \to t^{n+1}$}} \\
\underline{$k = 1$}:
Equating terms associated with $\mco(\veps^{-2})$ in the momentum update \eqref{eq:mom_ho_k}, we get,
    \[ a_{1,1} \left(\nabla \rho^{1}_{(0)} -\frac{\rho^{1}_{(0)}}{\rho_\mathrm{eq}} \nabla \rho_\mathrm{eq}\right) = 0 \quad (\mbox{type-A} \implies a_{1,1} \neq 0).\]
    Therefore, the density $\rho^{1}_{(0)}$ becomes well-prepared, but since the density at the previous time level $\rho_{(0)}^n$ is not well-prepared, we don't have the divergence-free condition for $\vq_{(0)}^1$.\\
    \underline{$k > 1$}: Let us assume that for the stages $k = 1,\ldots, s-1 $ the density is well-prepared. Now, equating like powers of $\mco (\veps^{-2})$ terms in the momentum update for the $s^{th}$ stage and considering the well-preparedness for $\rho^{k}$, $k=1,\ldots, s-1$ we get:
    \[a_{s,s} \left(\nabla \rho^{s}_{(0)} -\frac{\rho^{s}_{(0)}}{\rho_\mathrm{eq}} \nabla \rho_\mathrm{eq}\right) = 0 \quad (\mbox{type-A} \implies a_{s,s} \neq 0).\]
    Since the scheme is GSA we have $\rho^{n+1}_{(0)} = \rho^s_{(0)} = \rho_\mathrm{eq}$. But, we don't have the divergence-free condition for $\vq^{n+1}$. \\
    \underline{Second time-step $t^{n+1} \to t^{n+2}$} \\
    As obtained for the last iteration, analogously the intermediate densities are well-prepared, i.e.,
    \[\nabla \rho^{k}_{(0)} -\frac{\rho^{k}_{(0)}}{\rho_\mathrm{eq}} \nabla \rho_\mathrm{eq} = 0 , \quad \mbox{for } k = 1,\ldots s.
    \]
    \underline{$k=1$:} Considering the $\mco(\veps^0)$ terms in the mass update \eqref{eq:mass_ho_k}
    \[\frac{\rho_{(0)}^{1} - \rho_{(0)}^{n+1}}{\Dlt} = -a_{1,1} \dvg (\rho_{(0)}^1 \vu_{(0)}^1).\]
    Since $\rho^{n+1}$ is well-prepared $\rho_{(0)}^{1} = \rho_{(0)}^{n+1}$, and $a_{1,1} \neq 0$ (type-A), consequently we get
    \[\dvg (\rho_{(0)}^1 \vu_{(0)}^1) = 0.\]
    \underline{$k>1$:} Let us assume that for the stages $k = 1,\dots, s-1$ the following conditions hold true
    \[ \rho_{(0)}^k = \rho_{\mathrm{eq}} , \quad \mbox{and} \quad \dvg (\rho_{(0)}^k \vu_{(0)}^k) = 0.\]
    Now,  the $\mco(\veps^0)$ terms in the mass update \eqref{eq:mass_ho_k} for $k = s$ give:
     \[\frac{\rho_{(0)}^{s} - \rho_{(0)}^{n+1}}{\Dlt} = -a_{s,s} \dvg (\rho_{(0)}^s \vu_{(0)}^s), \quad \mbox{as} \  \dvg (\rho_{(0)}^k \vu_{(0)}^k) = 0, k =1, \ldots s-1. \]
     Using the same arguments as above and the fact that for a type-A scheme $a_{s,s} \neq 0$ we get the divergence condition $\dvg (\rho_{(0)}^s \vu_{(0)}^s) = 0$.
\end{proof}
\begin{remark}
    The above proof fails for type-CK schemes in general as $a_{11} = 0$ for this type, which hinders the projection of $\rho^1$ to the well-prepared space, i.e., the space of solutions of \eqref{eq:hydro_bal_an}.
\end{remark}\section{Space-time fully discrete scheme and its analysis}
\label{sec:fully-disc}
To obtain an implementable numerical scheme the spatial derivatives of the higher-order time semi-discrete scheme (Definition~\ref{def:ho_TSD}) need to be approximated. As understood from the scheme presented in the Introduction (Section~\ref{sec:intro}), much caution needs to be exercised to maintain the hydrostatic balance \eqref{eq:hydro_bal_an} also for non-zero $\veps$. In this section we present a second-order well-balanced FV discretization of the AP time semi-discrete scheme (Definition~\ref{def:ho_TSD}).

To this end we consider the following 1D form of the Euler equations \eqref{eq:iseuler_mass_ane}--\eqref{eq:iseuler_mom_ane},
\begin{equation}\label{eq:1d_isen_eul_pen}
    \Dt U + \pdxo \begin{pmatrix}
      0 \\
      F_2 (U)
    \end{pmatrix}+ \begin{pmatrix}
        0 \\
        \frac{S(U)}{\veps^2}
    \end{pmatrix} - \pdxo
    \begin{pmatrix}
        0 \\
        \frac{G_2 (U)}{\veps^2}
    \end{pmatrix} +
    \begin{pmatrix}
        0 \\
        \frac{\tilde{S}(U)}{\veps^2}
    \end{pmatrix} +
    \pdxo
    \begin{pmatrix}
        G_1(U) \\
        \frac{G_2(U)}{\veps^2}
    \end{pmatrix} -
    \begin{pmatrix}
        0 \\
        \frac{\tilde{S}(U)}{\veps^2}
    \end{pmatrix} = 0,
\end{equation}
\begin{align*}
U:= (\rho, q)^T, \ \ &F_2(U) := \frac{q^2}{\rho} + \frac{p}{\veps^2}, \ S(U):= \rho \pdxo \phi, \ \tilde{S}(U) := \frac{\rho}{\rho_{\mathrm{eq}}} \pdxo \rho_\mathrm{eq}, \\
&G_1(U) := q \mbox{ and }  \ G_2 (U) := \rho.
\end{align*}
Note that due to the addition and subtraction of the linear balance to the momentum two different types of sources are required to be discretized:
\begin{enumerate}
    \item $S(U)$: termed the gravitational source, and
    \item $\tilde{S}(U)$ : termed the penalization source.
\end{enumerate}
Here, we specify the scheme in 1D; it can be extended dimension by dimension to higher dimensions. Let $\Omega$ be the closed interval $[a,b]$, discretized via Cartesian mesh cells of lengths $\Delta x_1$ in $x_1$-direction.
For notational convenience, we introduce the finite difference operators $\delta$ and $\mu$, defined for any grid function $\omega_i$ as
\begin{equation}
    \delta_{x_1} \omega_{i} = \omega_{i + \frac{1}{2}} - \omega_{i - \frac{1}{2}}, \qquad \mu_{x_1} \omega_{i} = (\omega_{i + \frac{1}{2}} + \omega_{i - \frac{1}{2}}) / 2.
\end{equation}

\subsection{Gravitational source term discretization}
Following the approach introduced in \cite{Chandrashekar2017}, we first rewrite the gravitational source term $S(U)$ in \eqref{eq:1d_isen_eul_pen} as
\begin{equation}
    \rho \pdxo\phi= -p e^{-H(x_1)} \partial_x e^{H(x_1)},
    \;\; \mbox{where} \;
    H(x_1) = \frac{\gamma}{\gamma-1} \log \left (\frac{\gamma -1}{\gamma} \left (\beta - \phi(x_1) \right )\right ).
    \label{eq:source_transform}
\end{equation}
Here, $\beta = \frac{\gamma}{\gamma-1}\rho_\mathrm{eq}^{\gamma-1} + \phi$.
As a consequence of the above formulation, the gravitational source term on a finite volume cell $\Omega_i = [x_{1,i-\frac{1}{2}}, x_{1,i+\frac{1}{2}}]$ is approximated as
\begin{equation}
    S (U_i) \approx \mathcal{S}_i := -p_i e^{-H_i} \left (\frac{e^{H_{i+1/2}} - e^{H_{i-1/2}}}{\Delta x_1} \right ), \ \ \mbox{where} \ \ H_{i+1/2} = H(x_{1,i+1/2}),
    \label{eq:grav_sour_disc}
\end{equation}
and $p_i$ is a cell-centered approximation of the pressure.
\begin{lemma}
The source term discretization \eqref{eq:grav_sour_disc} is a second-order approximation of $\rho_i \partial_{x} \phi(x_i)$.
\end{lemma}
\begin{proof}
    The expression in \eqref{eq:grav_sour_disc} can be rewritten using
    \begin{equation} \label{eq:source_2ndo}
        \frac{(\gamma - 1)}{\gamma}p_i e^{-H_i} \left (\frac{e^{H_{i+\frac{1}{2}}} - e^{H_{i-\frac{1}{2}}}}{\Delta x_1} \right )
        =
        \frac{\gamma-1}{\gamma} p_i^{\frac{1}{\gamma-1}} \left ( \frac{ (\beta -\phi_{i+\frac{1}{2}})^{\frac{\gamma}{\gamma-1}} - (\beta - \phi_{i-\frac{1}{2}})^{\frac{\gamma}{\gamma-1}}}{\Delta x_1} \right ),
    \end{equation}
    where $\phi_{i+\frac{1}{2}} = \phi(x_{i+\frac{1}{2}})$. Now, plugging in the Taylor series expansion around $x_i$ for each of the terms in the above expression, we get
    \begin{equation*}
        \frac{ (\beta - \phi_{i+\frac{1}{2}})^{\frac{\gamma}{\gamma-1}} - (\beta - \phi_{i-\frac{1}{2}} )^{\frac{\gamma}{\gamma-1}}}{\Delta x_1} = -\frac{\gamma}{\gamma-1} (\beta -\phi_i)^{\frac{1}{\gamma-1}} \pdxo \phi_i + \mco(\Delta x_1^2).
    \end{equation*}
    Substituting the above expression in \eqref{eq:source_2ndo} we obtain
    \begin{equation*}
        -\frac{\gamma - 1}{\gamma}p_i e^{-H_i} \left (\frac{e^{H_{i+\frac{1}{2}}} - e^{H_{i-\frac{1}{2}}}}{\Delta x_1} \right ) = \rho_i \partial_x \phi_i + \mco(\Delta x_1^2).
    %    \qedhere
    \end{equation*}
\end{proof}
\subsection{Balance preserving variables}
\label{sec:bal-pres_var}
To obtain a second-order FV discretization while maintaining the steady state, caution needs to be exercised during reconstruction. With this goal, we notice that at a steady state where $u = 0$ we have
\begin{equation}
    p e^{-H} = \text{const}.
\end{equation}
Therefore, to maintain the steady state, we reconstruct the variable
\begin{equation}\label{eq:steady_state_var}
    \mathbf{\omega} = (pe^{-H}, u).
\end{equation}
The primitive variable $\rho$ can be obtained by first obtaining the pressure as
\begin{equation}
    p_{i+\frac{1}{2}}^- = e^{H_{i+\frac{1}{2}}} w_{1,i+\frac{1}{2}}^-,  \quad  p_{i+\frac{1}{2}}^+ = e^{H_{i+\frac{1}{2}}} w_{1,i+\frac{1}{2}}^+,
\end{equation}
followed by
\begin{equation}
    \varrho_{i+\frac{1}{2}}^- = (p_{i+\frac{1}{2}}^-)^{1/\gamma}, \quad \varrho_{i+\frac{1}{2}}^+ = (p_{i+\frac{1}{2}}^+ )^{1/\gamma}.
\end{equation}
\begin{remark}
    Note that a polynomial reconstruction of the variable in \eqref{eq:steady_state_var} preserves the steady state.
\end{remark}
\subsection{Finite volume scheme}
The fully-discrete higher-order FV scheme complementing the time semi-discrete AP scheme in Definition~\ref{def:ho_TSD} is obtained by combining a reconstruction of steady-state variables \eqref{eq:steady_state_var}, a Rusanov-type flux for the explicit part, central differences for the implicit part, and balance-preserving source term discretizations.
Next, we state the fully discrete scheme complementary to the semi-discrete scheme in Definition~\ref{def:ho_TSD} for the 1D penalized Euler equations \eqref{eq:1d_isen_eul_pen}.
\begin{definition}\label{def:fully_disc_AP}
On a finite volume cell $\Omega_i = [x_{1,i-1/2}, x_{1,i+1/2}]$ the $k^{th}$ stage of the $s$-stage fully discrete IMEX-RK scheme first computes the following variables explicitly:
\begin{align}
    \hat{\rho}^k_i = \rho^n_i &- \sum_{\ell = 1}^{k-1} a_{k, \ell} \nu_{x_1} \delta_{x_1} \mathcal{G}_{1,i}^{\ell},\\
    \hat{q}^k_i = q^n_i &- \sum_{\ell = 1}^{k-1} \tilde{a}_{k, \ell} \left(\nu_{x_1} \delta_{x_1}  \Big(\mathcal{F}_{2,i}^{\ell} - \frac{\mathcal{G}_{2,i}^\ell}{\veps^2} \Big) + \frac{\Delta t}{\veps^2} (\mathcal{S}_i^\ell + \mathcal{\tilde{S}}_i^\ell)\right) \notag \\
    & - \sum_{\ell = 1}^{k-1} \frac{a_{k,\ell}}{\veps^2} \left(\nu_{x_1} \delta_{x_1} \mathcal{G}_{2,i}^\ell - \Delta t \mathcal{\tilde{S}}^\ell_i\right).
\end{align}
Next, we solve the discrete linear elliptic problem for $\rho_i^k$:
\begin{equation}\label{eq:elliptic_fd}
    -\frac{\nu_{x_1}^2 a_{k,k}^2}{\veps^2} \delta_{x_1} \left(\delta_{x_1} \rho_i^{k}\right) + \frac{\nu_{x_1}^2a_{k,k}^2}{\veps^2} \delta_{x_1} \tilde{\mcal{S}}^{k}_i + \rho^{k}_i = \hat{\rho}^k_i - a_{k,k}\nu_{x_1}  \delta_{x_1} \hat{q}^{k}_i.
\end{equation}
Finally, the momentum at stage $k$, $q_i^k$, is updated explicitly as
\begin{equation}
    q_i^k = \hat{q}_i^k - \frac{a_{k,k}}{\veps^2} \left(\nu_{x_1} \delta_{x_1} \mathcal{G}_{2,i}^k - \Delta t \tilde{\mathcal{S}}_i^{k} \right).
\end{equation}
Here $\nu_{x_1} = \frac{\Delta t}{\Delta x_1}$ is the mesh ratio, and $\mathcal{F}_{2}$, $\mathcal{G}_{1}$, $\mathcal{G}_{2}$ and $\mathcal{S}$, $\mcst$ are respectively the numerical fluxes and the source term approximations defined as follows:
\begin{equation}\label{eq:num_flux}
    \begin{aligned}
        &\mathcal{F}_{2, i+1/2}^\ell := \frac{1}{2} \left(F_2(U^{\ell,+}_{i+1/2}) + F_2(U^{\ell,-}_{i+1/2})\right) -  \frac{\alpha^\ell_{i+1/2}}{2} \left(q_{i+1/2}^{\ell,+} - q_{i+1/2}^{\ell,-}\right),\\
        &\mathcal{G}_{1,i+1/2}^{\ell} := \frac{1}{2}(q^{\ell}_{i+1} + q^{\ell}_{i}), \qquad  \qquad  \qquad
        \mathcal{G}_{2,i+1/2}^{\ell} := \frac{1}{2}(\rho_i^\ell + \rho_{i+1}^\ell),
    \end{aligned}
\end{equation}
\begin{equation}\label{eq:num_source}
\mcal{S}^{\ell}_i:= -p^{\ell}_i e^{-H_i} \left (\frac{e^{H_{i+1/2}} - e^{H_{i-1/2}}}{\Delta x_1} \right ),
    \qquad \mcst_i^{\ell} := \frac{\rho_i^{\ell}}{\rho_{\mathrm{eq},i}} \frac{(\rho_{\mathrm{eq},i+1}-\rho_{\mathrm{eq},i-1})}{2 \Delta x_1}.
\end{equation}
Here for any variable $w$, we have denoted by $w_{i + 1/2}^{\pm}$ the interpolated states obtained using the piecewise linear reconstructions at the cell boundary $x_{i+1/2}$. The wave speeds of the explicit subsystem in the $x_1$-direction are
\begin{equation}
  \label{eq:wave_speed1}
  \alpha^{\ell}_{i
    +\frac{1}{2}}:=2\max\left(\left|q_{i+\frac{1}{2}}^{\ell,-}\big/\rho_{i+\frac{1}{2}}^{\ell,-}\right|,
    \left|q_{i+\frac{1}{2}}^{\ell,+}\big/\rho_{i+\frac{1}{2}}^{\ell,+}\right|\right).
\end{equation}
The momentum flux $F_2$ is approximated by a Rusanov-type scheme whereas the mass flux is approximated using central differences.
The eigenvalues of the Jacobian of the flux which is approximated by a Rusanov-type flux are $\lambda_{1}=0$,  $\lambda_{2}= 2 \frac{q}{\rho},$ and the CFL condition at time $t^n$ is given by
\begin{equation}
\label{eq:ep_CFL}
\Dlt\max_{i}\left(\left|\lambda_{2,i}\right|\big/\Delta x_1\right)=\nu,
\end{equation}
where $\nu<1$ is the given CFL number; see \cite{GP16} for details about the choice of the CFL conditions.
\end{definition}
\begin{remark}
    The reconstruction is first carried out for the hydrostatic variable $pe^{-H}$ and is then transformed to obtain the reconstruction of $\rho$; see Section~\ref{sec:bal-pres_var}.
\end{remark}
\subsection{Analysis of the fully-discrete scheme}
\begin{lemma}
\label{lemma:wb_pressure-source}
    The discretization of the source term to gradient pair $\pdxo (\frac{(q^\ell)^2}{\rho^{\ell}} + \frac{p^\ell}{\veps^2}) + \rho^\ell \frac{\pdxo \phi}{\veps^2}$ in Definition~\ref{def:fully_disc_AP} is well-balanced, i.e., if $(\rho^\ell_i, u^\ell_i) = (\rho_{\mathrm{eq},i}, 0)$, then the discretization $\nu_{x_1} \delta_{x_1} \mathcal{F}^{\ell}_{2,i} + \Delta t \mathcal{S}^\ell_i/\veps^2 = 0$.
\end{lemma}
\begin{proof}
     Consider the discretization of the momentum flux $\pdxo (\frac{q^2}{\rho} + \frac{p}{\veps^2}) $. First, note that the reconstruction preserves the equilibrium state achieved by the source gradient pair $\pdxo p^n + \rho^n \pdxo \phi$.
     Since $p_i e^{-H_i} = \text{const.}$,
     the reconstruction of the pressure term on the left can be written as
    \begin{equation}
        \left ( \omega^L_{i+1/2} \right )_1 = p_i e^{-H_i},
    \end{equation}
    as the diffusion term in the flux $\alpha_{2, i + \frac{1}{2}} (q_{i+\frac{1}{2}}^{n,+} - q_{i+\frac{1}{2}}^{n,-}) = 0$.
Note that this holds for any linear reconstruction. The pressure is therefore
    \begin{equation}
        p_{i+1/2}^- = e^{H_{i+1/2}} p_i e^{-H_i} = p_i e^{H_{i+1/2}- H_i} = p_{i+1/2}.
    \end{equation}
    The same applies for the right interface and for $i-1/2$, thus we end up with
    \begin{equation}
        \frac{1}{\Delta x_1}\left ( \frac{p^{+}_{i+1/2} + p^-_{i+1/2}}{2} - \frac{p^{+}_{i-1/2} + p^-_{i-1/2}}{2} \right ) = \frac{p_{i+1/2} - p_{i-1/2}}{\Delta x_1}.
    \end{equation}
    For the second term, i.e., the discretization of the source term, we notice that
    \begin{equation}
    -p_i e^{-H_i} \left (\frac{e^{H_{i+1/2}}-e^{H_{i-1/2}}}{\Delta x_1} \right ) = \frac{p_{i+1/2} - p_{i-1/2}}{\Delta x_1}.
    \end{equation}
    Thus, the discrete source gradient pair is well-balanced.
\end{proof}

\begin{lemma}\label{lemma:wb_density-source}
    The discretization of the source term-gradient pair $\pdxo \rho^{\ell} - \frac{\rho^{\ell}}{\rho_\mathrm{eq}}\pdxo \rho_\mathrm{eq}$ in Definition~\ref{def:fully_disc_AP} is well-balanced, i.e., if $(\rho^\ell_i, u^\ell_i) = (\rho_{\mathrm{eq},i}, 0)$ then the discretization $\nu_{x_1} \delta_{x_1} \mathcal{G}_{2,i}^\ell - \Delta t \tilde{\mathcal{S}_i^\ell} = 0$.
\end{lemma}
\begin{proof}
From \eqref{eq:num_flux} we have the term $\pdxo \rho^\ell$ is discretized as:
\begin{equation}\label{eq:dxrhon_wb}
    \nu_{x_1} \delta_{x_1} \mathcal{G}_{2,i}^\ell = \nu_{x_1}\frac{(\rho^\ell_{i+1} + \rho^\ell_{i}) - (\rho^\ell_{i} + \rho^\ell_{i-1})}{2} = \frac{\Delta t}{\Delta x_1} \frac{(\rho_{\mathrm{eq},i+1} - \rho_{\mathrm{eq},i-1})}{2}.
\end{equation}
Now, looking at $\Delta t \tilde{\mcal{S}}^\ell_{i}$ from \eqref{eq:num_source} we get
\begin{align}\label{eq:si_wb}
    \Delta t \mcst_i^{\ell} = \Delta t \frac{\rho^{\ell}_i}{\rho_{\mathrm{eq},i}} \left(\frac{\rho_{\mathrm{eq},i+1} - \rho_{\mathrm{eq}, i-1}}{2 \Delta x_1} \right) = \Delta t \frac{\rho_{\mathrm{eq},i+1} - \rho_{\mathrm{eq}, i-1}}{2 \Delta x_1}.
\end{align}
Combining \eqref{eq:dxrhon_wb} and \eqref{eq:si_wb} we get
$\nu_{x_1} \delta_{x_1} \mathcal{G}_{2,i}^\ell - \Delta t \mcst_i^{\ell}= 0$.
\end{proof}

\begin{lemma}\label{lemma:wb_elliptic}
    The discrete elliptic problem \eqref{eq:elliptic_fd} for initial data in equilibrium, i.e., $(\rho^n_i, u^n_i) = (\rho_{\mathrm{eq},i}, 0)$, and for  Dirichlet ($\rho_{\mathrm{eq}}$), periodic, or Neumann extrapolation boundary conditions, has the unique solution $\rho_{\mathrm{eq},i}$.
\end{lemma}
\begin{proof}
    We prove the result for the first stage, i.e., $k = 1$. For the other stages, the result follows by induction and well-balancedness of the flux-source discretizations.

    Since $u^n_i = 0$ and $\rho^n_i = \rho_{\mathrm{eq},i}$, Lemmas~\ref{lemma:wb_pressure-source} and \ref{lemma:wb_density-source} yield $\hat{q}^{1}_i = 0$. Hence, the elliptic problem \eqref{eq:elliptic_fd} becomes
    \begin{equation*}
        \begin{aligned}
           & -\frac{\nu_{x_1}^2 a_{1,1}^2}{\veps^2} \delta_{x_1} \left(\delta_{x_1} \rho_i^{1}\right) +\frac{\nu_{x_1} \Dlt a_{1,1}^2}{\veps^2} \delta_{x_1} \tilde{\mcal{S}}^{1}_i + \rho^{1}_i = \rho_{\mathrm{eq},i} \\
            \implies & -\frac{\nu_{x_1}^2 a_{1,1}^2}{\veps^2} \delta_{x_1} \left(\delta_{x_1} \rho_i^{1}\right) +\frac{\nu_{x_1}^2 a_{1,1}^2}{\veps^2} \delta_{x_1} \left( \frac{\rho_i^{1}}{\rho_{\mathrm{eq},i}} \delta_{x_1} \rho_{\mathrm{eq},i} \right) + \rho^{1}_i = \rho_{\mathrm{eq},i}.
        \end{aligned}
    \end{equation*}
If $\rho^{1}_i = \rho_{\mathrm{eq},i}$, then
\[ -\frac{\nu_{x_1}^2 a_{1,1}^2}{\veps^2} \delta_{x_1} \left(\delta_{x_1} \rho_i^{1}\right) +\frac{\nu_{x_1}^2 a_{1,1}^2}{\veps^2} \delta_{x_1} \left( \frac{\rho_i^{1}}{\rho_{\mathrm{eq},i}} \delta_{x_1} \rho_{\mathrm{eq},i} \right)  = 0. \]
Thus, $\rho^{1}_i = \rho_{\mathrm{eq},i}$ solves the linear system above.
Imposing boundary conditions yields the uniqueness of the solution.
\end{proof}
\begin{proposition}
    The fully-discrete numerical scheme in Definition~\ref{def:fully_disc_AP} is well-balanced, i.e., if $(\rho^n_i, u^n_i) = (\rho_{\mathrm{eq},i}, 0) \implies (\rho^{n+1}_i, u^{n+1}_i) = (\rho_{\mathrm{eq},i}, 0)$.
\end{proposition}
\begin{proof}
Lemmas~\ref{lemma:wb_pressure-source} and \ref{lemma:wb_density-source} ensure that the flux and source term discretizations maintain the steady state at every intermediate stage, i.e.,
$u^k_i = 0$ for $k = 1, \ldots s$.
Moreover, the uniqueness of $\rho_{\mathrm{eq},i}$ as a solution to the linear elliptic problem \eqref{eq:elliptic_fd} ensures that
$\rho^k_i = \rho_{\mathrm{eq},i}$, for $k = 1, \ldots, s$.
Since we consider GSA IMEX-RK schemes we have
$\rho^{n+1}_i = \rho^s_i = \rho_{\mathrm{eq},i}$ and $u^{n+1}_i = 0$.
\end{proof}

\section{Numerical case studies}
\label{sec:numerical_results}
In this section we present the results of numerical case studies. The tests are
carried out for the penalized IMEX-WB-AP scheme in Definition~\ref{def:fully_disc_AP}, to be referred to as WB-AP scheme. The additive IMEX scheme of choice is the DP2-A(2,4,2) from \cite{DP13} given in \eqref{eq:dp2a2}.
For all the test cases, $\beta$ in \eqref{eq:dp2a2} is fixed at 0.7, until stated otherwise. The goals of the numerical case studies carried out are to verify the following properties of the penalized IMEX-WB-AP scheme:
\begin{itemize}
    \item well-balancedness;
    \item uniform second-order convergence in the asymptotic regime;
    \item asymptotic stability, i.e., time steps independent of $\veps$;
    \item employability for a wide range of $\veps$.
\end{itemize}
These properties assert the robustness of the newly developed scheme. For the NWB scheme, to be referred to as NWB-AP scheme, the same IMEX-AP scheme as Definition~\ref{def:fully_disc_AP} is used barring the hydrostatic reconstruction \eqref{eq:steady_state_var} and a simple second-order central difference for the sources instead of the WB source discretization in \eqref{eq:grav_sour_disc}.
All source code required to reproduce the numerical experiments is available online~\cite{artiano2026asymptoticRepo}
\begin{equation}
\label{eq:dp2a2}
\begin{aligned}
   \begin{array}{c|c c c c}
    0	  & 0		& 0	    & 0	     & 0\\
    0     & 0		& 0	    & 0	     & 0\\
    1     & 0		& 1     & 0	     & 0\\
    1     & 0       & 1/2 & 1/2 & 0\\
    \hline
    0     & 0       & 1/2 & 1/2  & 0\\
  \end{array}
  \qquad &
  \begin{array}{c|c c c c}
    \beta & \beta  & 0	         & 0	          & 0  \\
    0	     & -\beta & \beta      & 0	          & 0  \\
    1	     & 0         & 1 - \beta  & \beta       & 0  \\
    1	     & 0         & 1/2         & 1/2 - \beta & \beta  \\
    \hline
             & 0         & 1/2         & 1/2 - \beta & \beta
  \end{array} \\
  \mbox{DP2-A(2,4,2)}, & \mbox{\ type-A,  GSA \cite{DP13}.}
  \end{aligned}
\end{equation}
\subsection{Hydrostatic steady state and its perturbation}
\label{ssec:hydro_stat}
We consider the hydrostatic steady state \eqref{eq:rho_eq} with $C_0 = 1$ and add a perturbation of magnitude $\zeta$ to it \cite{Chandrashekar2015}. The initial data reads:
\begin{equation}\label{eq:wb_1D_mD}
\begin{aligned}
    \rho(0,\vx) &= \rho_{(0)}(\vx)+ \zeta  \psi(\vx), \qquad \vu(0, \vx) = 0, \\
    \rho_{(0)}(\vx)&:=\left(1 - \frac{\gamma - 1}{\gamma} \phi(\vx) \right)^{\frac{1}{\gamma - 1}}.
\end{aligned}
\end{equation}
In what follows, different tests are carried out with the above initial conditions, in 1D and 2D, with different choices of $\phi$, $\psi$, and $\zeta$.
For all the test cases here we apply no-flux boundary conditions.
As the initial velocity is zero, a time step $\Delta t = 0.5  \Delta x $ is chosen for all values of $\veps$, if not stated otherwise.

\subsubsection*{Case 1: Well-balancing test}
In this 1D test case the spatial domain is $[0,1]$. We consider three different choices for gravitational potentials: $\phi(x_1) = x_1, x_1^2$, and $\sin(2 \pi x_1)$.
The parameters for this test are $\zeta = 0.0$ (the steady state) and $\gamma = 1.4$.
The spatial domain is discretized by 100 mesh points and as the initial velocity is zero, a time step $\Delta t = 0.5  \Delta x = \mathcal{O}(10^{-2})$ is chosen for all values of $\veps$.

Table~\ref{table:1d_wb} shows the $L^2$-errors for $\veps$ ranging from $1$ (compressible regime) to $10^{-4}$ (anelastic regime) at a very long time $T=10.0$.
The newly developed penalized IMEX-WB-AP scheme maintains the steady state with very high precision, for all the different choices of potential, irrespective of the flow regime determined by $\veps$.
\begin{table}[htb]
\begin{center}
  \caption{Well-balancing test (Case 1): $L^2$ errors in density and velocity for hydrostatic state at T = 10.0.}
   \vspace{-2mm}
  \label{table:1d_wb}
  {\setlength{\extrarowheight}{1.0mm}
  \begin{tabular}[10pt] {
|p{0.8cm}||p{1.4cm}|p{1.4cm}||p{1.4cm}|p{1.4cm}||p{1.4cm}|p{1.4cm}|}
    \hline
     & \multicolumn{2}{c||}{$\phi(x_1)=x_1$}
& \multicolumn{2}{c||}{$\phi(x_1)=x_1^2$}
& \multicolumn{2}{c|}{$\phi(x_1)=\sin(2 \pi x_1)$} \\
    \hline
    $\veps$& $L^2$ errors $\rho$ & $L^2$ errors $u_1$ & $L^2$ errors $\rho$ & $L^2$ errors $u_1$
    & $L^2$ errors $\rho$ & $L^2$ errors $u_1$ \\
    \hline
    $1.0$	& 4.56E-30	& 3.79E-30      & 8.42E-30 & 1.07E-30      &  9.08E-30 & 3.03E-30\\
    $10^{-1}$&1.02E-28	& 3.24E-28  & 1.20E-28 & 9.44E-29 & 1.03E-28 & 3.06E-28  \\
    $10^{-2}$	& 1.83E-27	& 3.28E-26  & 5.84E-27 & 6.32E-27 & 6.19E-27 & 4.65E-26 \\
    $10^{-3}$	& 3.00E-25	& 3.83E-24  & 3.32E-26 & 1.98E-24 & 4.78E-25 & 1.04E-23 \\
    $10^{-4}$	& 4.64E-23	& 3.80E-22  &  5.60E-23 & 1.10E-22 & 1.21E-22 & 5.14E-22 \\
    \hline
 \end{tabular}}
  \end{center}
  \vspace{-4mm}
\end{table}

\subsubsection*{Case 2: Perturbation of steady state in 1D}
Here we perturb the hydrostatic steady state $\rho_{\mathrm{eq}}$ \eqref{eq:rho_eq} in 1D, by choosing non-zero values of $\zeta$. The gravitational potential $\phi(x_1) = x_1$, $\psi(x_1) = e^{-100(x_1 - 0.5)^2}$, and $\gamma = 1.4$ in \eqref{eq:wb_1D_mD}. The spatial domain is $[0,1]$, which is discretized by 100 mesh points.

Figure~\ref{fig:eps1_den_per} shows the evolution of the density perturbation scaled by the perturbation parameter $\zeta = 10^{-3}$ for $\veps = 1.0$.
Figure~\ref{fig:eps1_den_per}(Left) shows results for the WB-AP scheme and NWB-AP schemes for $\zeta = 10^{-3}$.
Figure~\ref{fig:eps1_den_per}(Middle) shows the same for $\zeta = 10^{-5}$.
Both plots convey that the WB-AP scheme is able to reduce the initial perturbation, whereas the NWB-AP scheme starts developing oscillations and/or increases the perturbation magnitude.
Figure~\ref{fig:eps1_den_per}(Right) shows a zoom of the NWB-AP results from Figure~\ref{fig:eps1_den_per}(Middle), elucidating the well-balanced and AP damping effects of the WB-AP penalized-IMEX scheme in Definition~\ref{def:fully_disc_AP}.

\begin{figure}[htb]
  \centering
  \includegraphics[width = \linewidth]{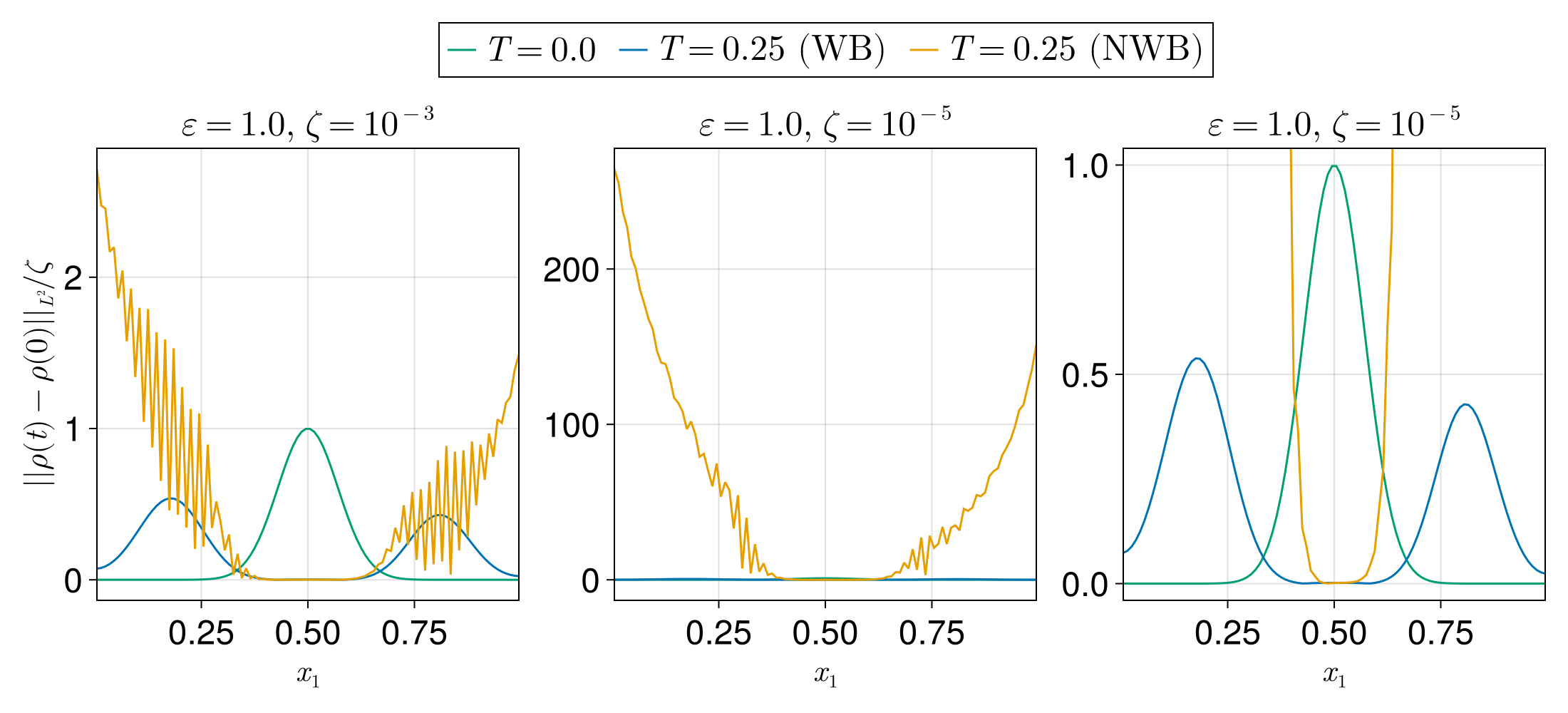}
  \vspace{-8mm}
  \caption{Perturbation of steady state in 1D: Plots of density perturbations for WB-AP and NWB-AP schemes for $\veps=1.0$. Left: $\zeta = 10^{-3}$; Middle: $\zeta = 10^{-5}$; Right: zoom of the middle panel.}
  \label{fig:eps1_den_per}
\end{figure}
%\begin{figure}[htb]
%  \centering
%  \includegraphics[height=0.144\textheight]{1D_pert_hydro/pert_wb_eps1m3.png}
%  \includegraphics[height=0.144\textheight]{1D_pert_hydro/pert_wb_eps1m5.png}
%  \includegraphics[height=0.144\textheight]{1D_pert_hydro/pert_wb_eps1m5_.png}
%  \vspace{-6mm}
%  \caption{Perturbation of steady state in 1D: Plots of density perturbations for WB-AP and NWB-AP schemes for $\veps=1.0$. Left: $\zeta = 10^{-3}$; Middle: $\zeta = 10^{-5}$; Right: zoom of the middle panel.}
%  \label{fig:eps1_den_per}
%\end{figure}

Figure~\ref{fig:epsm1-m3_den_per}(Left) plots the same profiles as Figure~\ref{fig:eps1_den_per}, but for $\veps = 10^{-1}$. Similarly, Figure~\ref{fig:epsm1-m3_den_per}(Middle) shows the same for $\veps = 10^{-2}$ and Figure~\ref{fig:epsm1-m3_den_per}(Right) for $\veps = 10^{-3}$.
For all values of $\veps$ the perturbation parameter $\zeta$ is set to be $\veps^2$. In Figure~\ref{fig:epsm1-m3_den_per}(Left) even though the NWB scheme has a similar magnitude of perturbation after T = 0.25, it is oscillatory in nature. Whereas for $\veps = 10^{-2}$ in Figure~\ref{fig:epsm1-m3_den_per}(Middle) the perturbation for the NWB scheme inflates, but the WB scheme is able to damp the initial perturbation. Figure~\ref{fig:epsm1-m3_den_per}(Right) shows that the WB scheme damps the perturbation for $\veps = 10^{-3}$, in contrast to the NWB scheme for which the solution blows up.
\begin{figure}[htbp]
 \vspace{-3mm}
  \centering
  \includegraphics[width = \linewidth]{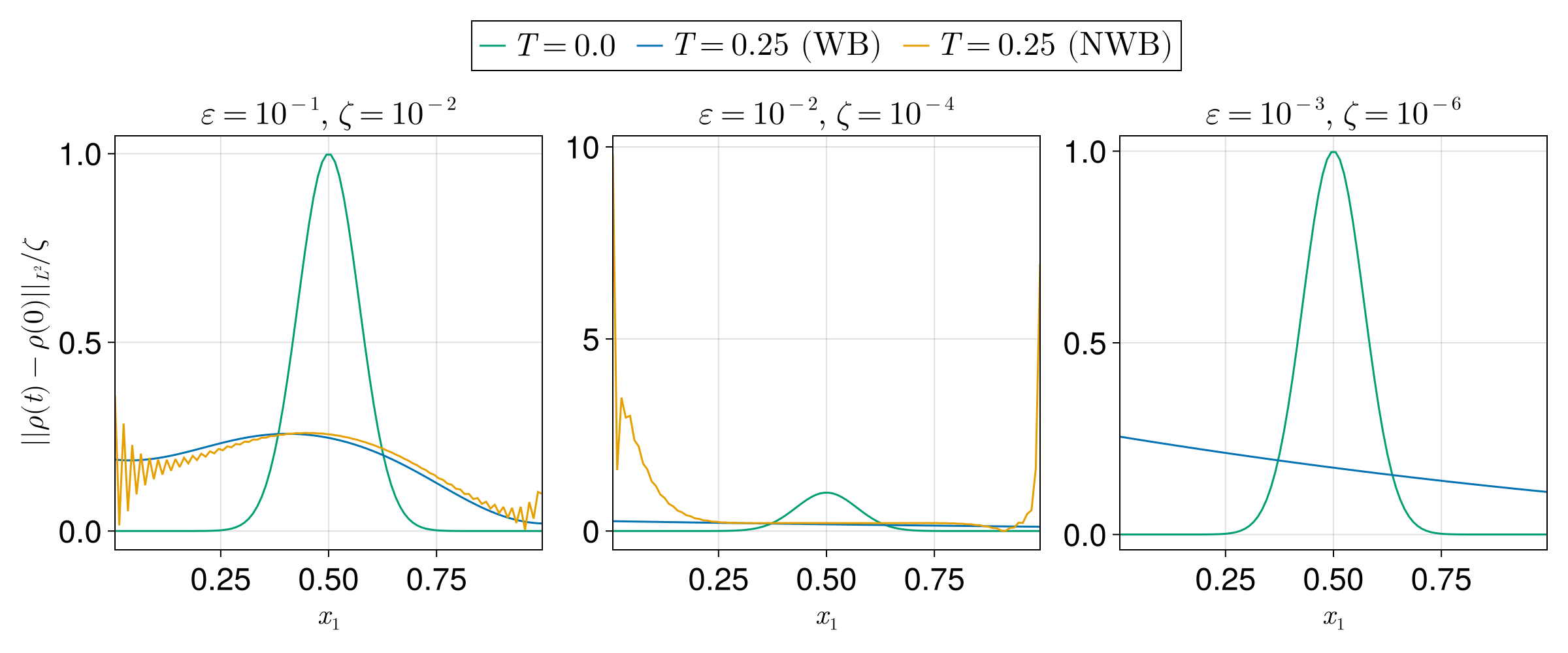}
  \vspace{-8mm}
  \caption{Perturbation of steady state in 1D: Plots of density perturbations for WB-AP and NWB-AP, with $\zeta = \veps^2$, $\veps = 10^{-1}, 10^{-2}, 10^{-3}$.}
  \label{fig:epsm1-m3_den_per}
  \end{figure}
%\begin{figure}[htbp]
% \vspace{-3mm}
%  \centering
%  \includegraphics[height=0.144\textheight]{1D_pert_hydro/pert_wb_epsm1.png}
%  \includegraphics[height=0.144\textheight]{1D_pert_hydro/pert_wb_epsm2.png}
%  \includegraphics[height=0.144\textheight]{1D_pert_hydro/pert_wb_epsm3.png}
%  \vspace{-6mm}
%  \caption{Perturbation of steady state in 1D: Plots of density perturbations for WB-AP and NWB-AP, with $\zeta = \veps^2$, $\veps = 10^{-1}, 10^{-2}, 10^{-3}$.}
 % \label{fig:epsm1-m3_den_per}
 % \end{figure}

\subsubsection*{Case 3: Asymptotic Order of Convergence (AOC)}
We compute the convergence order by considering the steady solution as the exact solution to obtain the asymptotic order of convergence \cite{SBG+99, AS20}.
Let $U_{\veps}(t^n)$ be the solution of Euler equations \eqref{eq:iseuler_mass_ane}--\eqref{eq:iseuler_mom_ane}, and $U_{(0)}(t^n)$ be the solution of the asymptotic limit system \eqref{eq:ane_lim_div}--\eqref{eq:ane_lim_mom}, at time $t^n$.
Let $U^n_{\veps}$ be the numerical solution obtained using the second-order numerical scheme developed here at the same time $t^n$.
Then, we expect the following estimates:
\begin{align*}
    U_{\veps}^n - U_{(0)}^n &= U_{\veps}^n - U_{\veps}(t^n) + U_{\veps}(t^n) - U_{(0)}(t^n) + U_{(0)}(t^n) - U_{(0)}^n \\
    &= \mco(\Dlt^2) + \mco(\veps) + \mco(\Dlt^2).
\end{align*}
Therefore, as $\veps \to 0$ we get from the above:
\[U_{\veps}^n - U_{(0)}^n = \mco(\Dlt^2).\]
The initial data is as in \eqref{eq:wb_1D_mD} in 1D with $x_1 \in [0,1]$, $\zeta = \veps$, $\phi(x_1) = x_1$, $\gamma = 1.4$.
The computations are carried out for successively refined meshes of $20, 40, 80, 160$ and $320$ mesh points with $\Delta t = 0.1 \times \Delta x_1 $, for a fixed final time T = 3.0 and for $\veps \in \{10^{-4},10^{-5}, 10^{-6}\}$.
Table~\ref{table:AOC} shows that the designed scheme achieves optimal second-order convergence rates while $\Dlt$ is kept independent of the choice of $\veps$, i.e., the scheme is uniformly asymptotically second-order accurate with respect to $\veps$.
\begin{table}\begin{center}
\caption{AOC for the penalized IMEX-WB-AP scheme for $\veps = 10^{-4}, 10^{-5}, 10^{-6}$.}
 \vspace{-4mm}
\begin{tabular}[t] {
|p{0.7cm}||p{1.8cm}|p{0.8cm}||p{1.8cm}|p{0.8cm}||p{1.8cm}|p{0.8cm}|}
    \hline
    & \multicolumn{2}{c||}{$\veps = 10^{-4}$}   & \multicolumn{2}{c||}{$\veps = 10^{-5}$}
    & \multicolumn{2}{c|}{$\veps = 10^{-6}$}\\
    \hline
    $N$& $L^2$ error $u_{1}$ & AOC & $L^2$ error $u_{1}$ &AOC
    & $L^2$ error $u_{1}$ & AOC \\
    \hline
    20	& 4.7880E-08 &      & 4.6743E-09 & &2.3970E-09 & \\

    40	& 1.2031E-08 & 1.99 & 1.2250E-09 & 1.93 & 6.5109E-10	& 1.88 \\

    80	& 3.1637E-09 & 1.92 & 3.1538E-10 & 1.95 & 1.6945E-10	& 1.94\\

    160	& 7.9766E-10 & 1.98 & 8.0725E-11 & 1.96 & 4.3404E-11	& 1.96\\

    320	& 2.0709E-10 & 1.94 & 2.1568E-11 & 1.90 & 1.1526E-11	& 1.91\\
    \hline
 \end{tabular}
   \label{table:AOC}
  \end{center}
\end{table}

\subsubsection*{Case 4: Perturbation of 2D steady state}
Next, we perturb the hydrostatic steady state $\rho_{\mathrm{eq}}$ \eqref{eq:rho_eq} in 2D by choosing $\zeta \ne 0$ in \eqref{eq:wb_1D_mD}.
The gravitational potential $\phi(x_1, x_2) = x_1 + x_2$, $\psi(x_1, x_2) = e^{-100[(x_1 - 0.3)^2 + (x_2 - 0.3)^2 ]}$, $\gamma = 1.4$, and the spatial domain $[0,1] \times [0,1]$ is divided into $50 \times 50$ mesh points.
For $\veps=1.0$, $\Dlt = 0.1 \times \Delta x$.
For $\veps=10^{-2}$, $\Dlt = 0.01 \times \Delta x$; this is to take into account the small final time.

Figure~\ref{fig:eps1-zm1_den_per} shows the density perturbation $|\rho(T) - \rho_{(0)}|$ for $\veps=1$ and different values of $\zeta$ at time T = 0.05.
Figure~\ref{fig:eps1-zm1_den_per} (Top-Left) and (Top-Right), respectively, show the plots obtained using the WB-AP scheme and the NWB-AP scheme for $\zeta = 10^{-1}$.
The performance of both schemes is similar for a large perturbation ($\zeta = 10^{-1}$).
However, Figure~\ref{fig:eps1-zm1_den_per} (Bottom-Left) and (Bottom-Right) show that the NWB-AP scheme is not able to resolve a smaller perturbation ($\zeta = 10^{-3}$), contrary to the WB-AP scheme.
\begin{figure}[htbp]
    \centering
    \includegraphics[width=\linewidth]{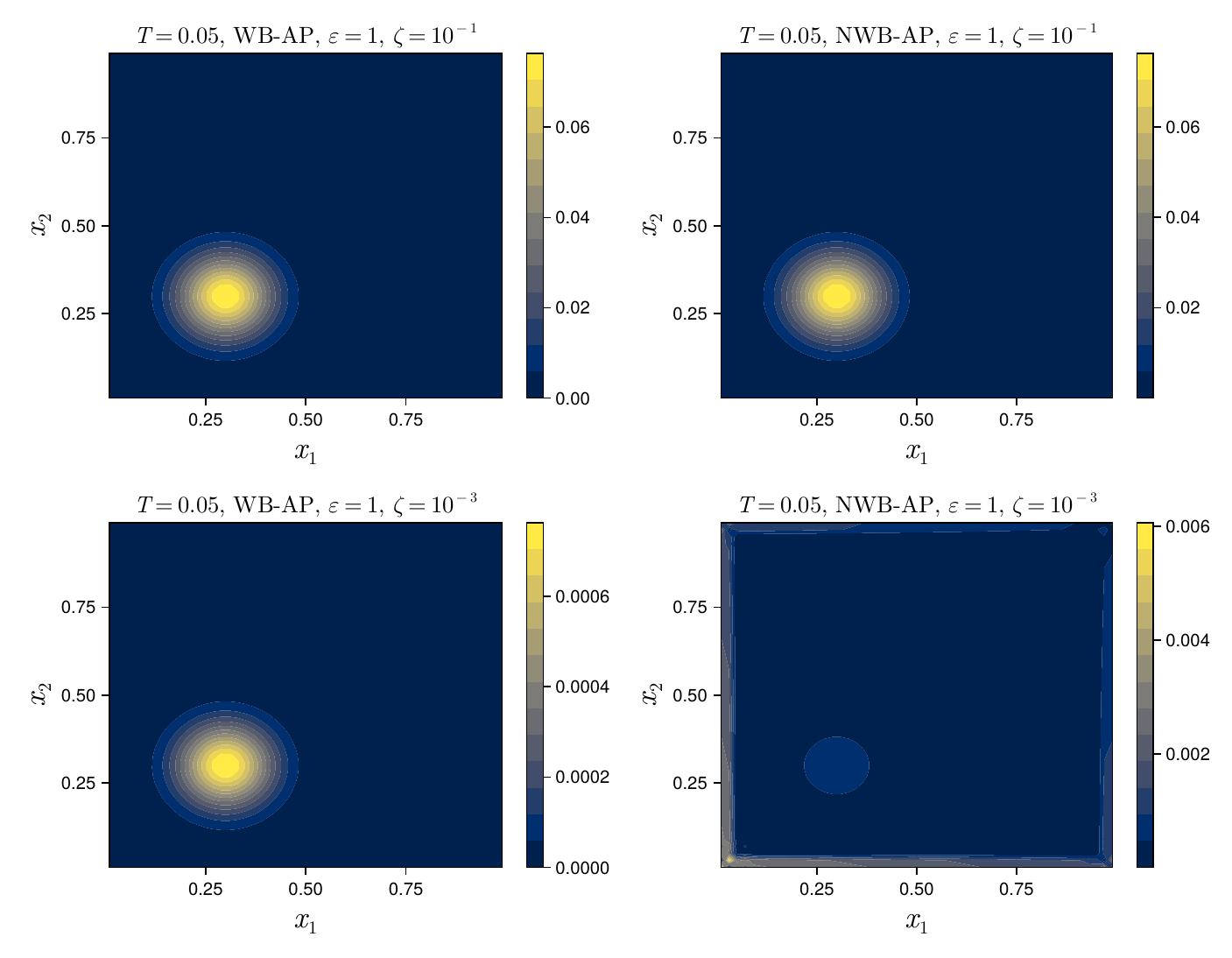}
    \vspace{-7mm}
  \caption{Perturbation of 2D steady state: Density perturbation profiles for $\veps=1.0$ and initial perturbation of the order $\mco({10^{-1}})$ and $\mco({10^{-3}})$ for the NWB-AP and WB-AP schemes.}
  \label{fig:eps1-zm1_den_per}
\end{figure}
%\begin{figure}[htbp]
% \vspace{-3mm}
%  \centering
%  \includegraphics[height=0.23\textheight]{2D_pert_hydro/eps1_zeta0p1_nwb.png}
%  \includegraphics[height=0.23\textheight]{2D_pert_hydro/eps1_zeta0p1.png}
%  \includegraphics[height=0.23\textheight]{2D_pert_hydro/eps1_zeta0p3_nwb.png}
%  \includegraphics[height=0.23\textheight]{2D_pert_hydro/eps1_zeta0p3.png}
%  \vspace{-5mm}
%  \caption{Perturbation of 2D steady state: Density perturbation profiles for $\veps=1.0$ and initial perturbation of the order $\mco({10^{-1}})$ and $\mco({10^{-3}})$ for the NWB-AP and WB-AP schemes.}
%  \label{fig:eps1-zm1_den_per}
%\end{figure}
Figure~\ref{fig:epsm1-m2_den_per} (Left) and (Right), respectively, show the plots obtained using the WB-AP scheme for $\veps = 10^{-1}$ (T = 0.005) and $10^{-2}$ (T = 0.001), respectively, and $\zeta = \veps^2$.
The WB-AP scheme is able to resolve even very small perturbations for small times.
\begin{figure}
    \centering
    \includegraphics[width=\linewidth]{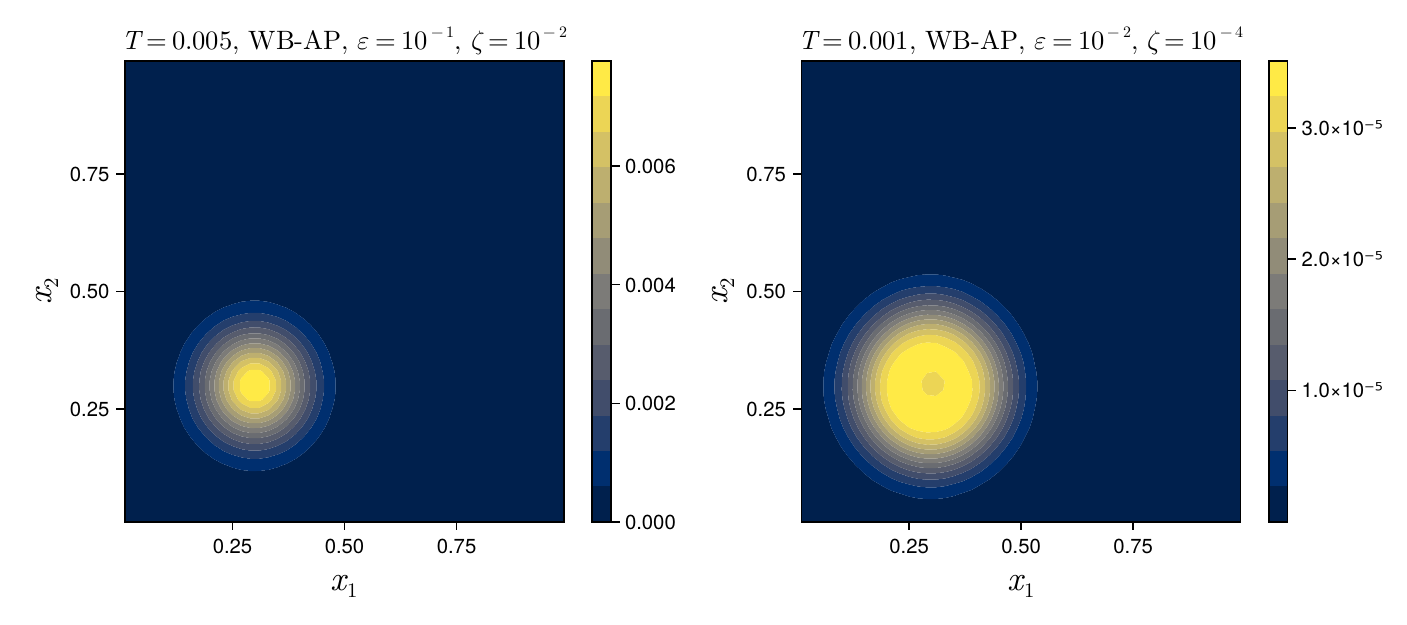}
    \vspace{-7mm}
  \caption{Perturbation of 2D steady state: Density perturbation profiles for $\veps=10^{-1}$ (Left) and $10^{-2}$ (Right) with an initial perturbation of order $\mco(\veps^2)$ for the WB-AP scheme.}
  \label{fig:epsm1-m2_den_per}
\end{figure}
%\begin{figure}[htbp]
% \vspace{-3mm}
%  \centering
%  \includegraphics[height=0.23\textheight]{2D_pert_hydro/epsm1_zeta.png}
%  \includegraphics[height=0.23\textheight]{2D_pert_hydro/epsm2_zeta.png}
%  \vspace{-6mm}
%\end{figure}
\subsection{Riemann problem}
We consider a Riemann problem with the initial condition considered in \cite{CDS26} and adapted from \cite{DT11}, given as follows:
\begin{equation}
\begin{aligned}
\rho (0, x_1) &=  1, \qquad &q(0, x_1) &= 1, \quad &x& \in [0, 0.25] \cup [0.75, 1], \\
\rho (0, x_1) &=  1 + \veps^2, &\qquad q(0, x_1) &= 1,  \quad &x& \in (0.25, 0.75].
\end{aligned}
\end{equation}
The isentropic constant is $\gamma = 2$ and the scaled Mach number is $\veps = 0.3$.
Extrapolation boundary conditions are imposed and the final time $T$ is $0.01$.
The reference solution is obtained by a standard first-order explicit scheme (as described in \cite{DT11}) on a fine mesh with $N = 10^4$ points and a CFL number $0.1$.
The penalized IMEX-WB-AP scheme in Definition~\ref{def:fully_disc_AP} is subjected to a coarser mesh of  $N = 10^3$ mesh points with CFL number $0.9$.

Figure~\ref{fig:riemann} shows the comparison of numerical solution profiles for density and momentum between the reference solution and the penalized IMEX-WB-AP scheme.
We observe that the scheme is capable of accurately capturing the shock structures, although the reference solution is computed on a finer mesh which naturally exhibits sharper gradients in the shock regions.
 \begin{figure}[htbp]
  \centering
  \includegraphics[width = \linewidth]{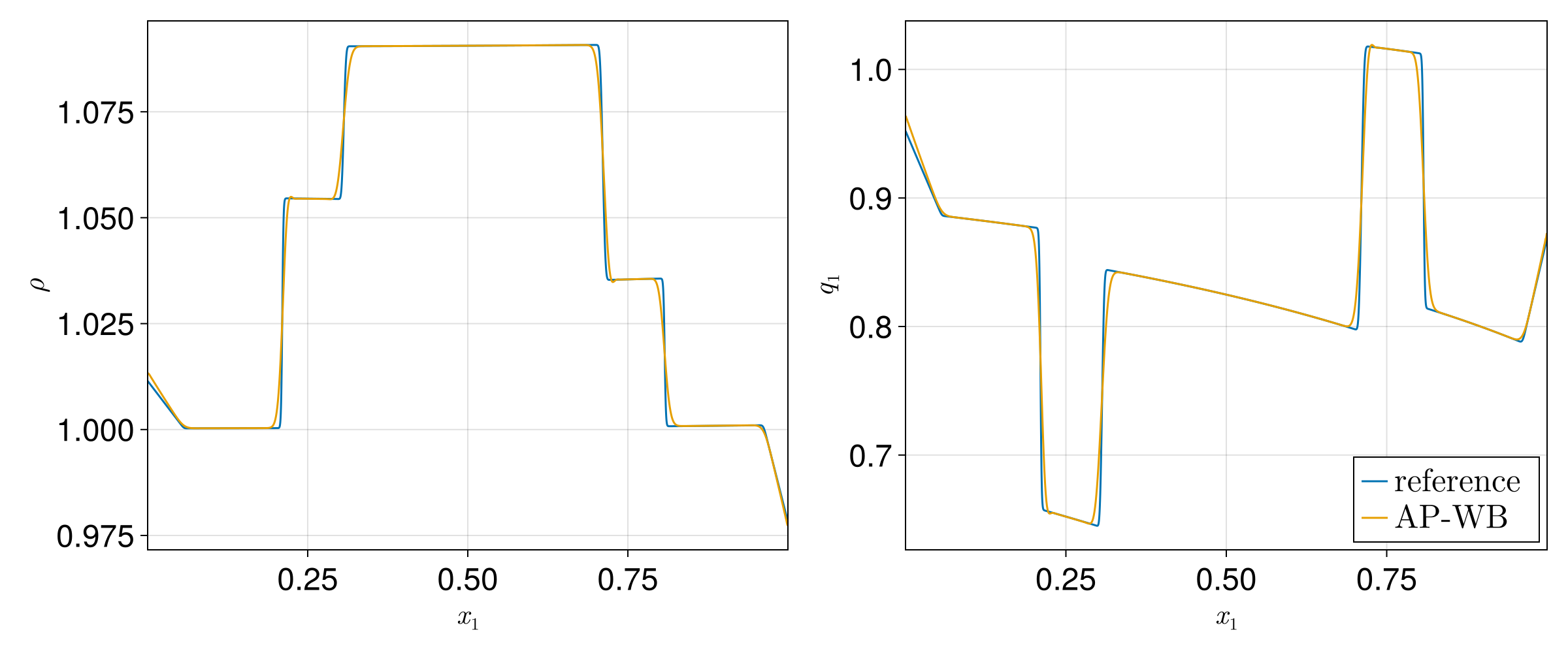}
   \vspace{-7mm}
  \caption{Riemann problem: $\veps = 0.3$. Left: density $\rho(T, x_1)$. Right: momentum $\rho u_1(T, x_1)$.  Comparison between penalized IMEX-WB-AP scheme ($10^3$ mesh points) and reference solution ($10^4$ mesh points) at $T = 0.01$. Here, $\beta = 0.5$ in \eqref{eq:dp2a2}.}
  \label{fig:riemann}
\end{figure}
 %\begin{figure}[htbp]
  %\centering
  %\includegraphics[height=0.23\textheight]{riemann/rho.png}
  %\includegraphics[height=0.23\textheight]{riemann/vel.png}
  % \vspace{-4mm}
  %\caption{Riemann problem: $\veps = 0.3$. Left: density $\rho(T, x_1)$. Right: momentum $\rho u_1(T, x_1)$.  Comparison between penalized IMEX-WB-AP scheme ($10^3$ mesh points) and reference solution ($10^4$ mesh points) at $T = 0.01$. Here, $\beta = 0.5$ in \eqref{eq:dp2a2}.}
  %\label{fig:riemann}
%\end{figure}
\subsection{2D vortex}
We consider a stationary vortex problem inspired by \cite{Zenk2018}, which serves as a low-Mach benchmark for numerical schemes.
% To assess these properties the Mach number, the kinetic energy dissipation and vorticity of flow are parameters which are taken into consideration.

The problem is set up in a square domain $[0,1] \times [0,1]$, the isentropic parameter $\gamma = 2$ and $\phi(\vx) = \phi(r) = r^2$ with radius $r = \sqrt{(x - 0.5)^2 + (y - 0.5)^2}$.
The vortex is centered at $(0.5, 0.5)$ and the initial condition reads:
\begin{equation}
\begin{aligned}
\rho(0, \vx) &= 1 + \frac{\veps^2}{2} \int_0^r \frac{u_{\theta}^2 (s)}{s} \ ds - \frac{1}{2} r^2\\
u (0, \vx) &= u_{\theta} (r) (y - 0.5) / r, \quad
v (0, \vx) = -u_{\theta} (r) (x - 0.5) / r,
\end{aligned}
\end{equation}
where $u_{\theta}$ is defined as
\vspace{-5mm}
\begin{equation*}
u_{\theta} (r) = \begin{cases}
a_1 r, \qquad &\mbox{if } r \leq r_1 \\
a_2 + a_3 r , \qquad &\mbox{if } r_1 < r \leq r_2 \\
0, \qquad &\mbox{otherwise.}
\end{cases}
\end{equation*}
The parameters $r_1 = 0.2$ and $r_2 = 0.4$ are the inner and the outer radii. $\bar{a}, a_1, a_2$ and $a_3$ are constants, where $\bar{a} = 0.1$, $a_1 = \bar{a}/r_1$, $a_2 = -\bar{a} r_2/(r_1 - r_2)$, and $a_3 = \bar{a}/(r_1 - r_2)$.

 Figure~\ref{fig:t0_mach} shows the initial Mach number $(M = \sqrt{(u^2 + v^2)/(\gamma p/\rho)})$ plot for the stationary vortex, which is independent of $\veps$.
 The vortex is allowed to evolve until a final time $T = 1.0$ (one rotation) for $\veps$ ranging from $10^{-1}$ to $10^{-4}$ on a grid of $200 \times 200$ mesh points and a fixed $CFL = 0.22$.
% \begin{figure}[htbp]
% \vspace{-4mm}
%  \centering
%\includegraphics[height=0.24\textheight]{2D_vortex/mach_0.png}
%  \caption{2D vortex: pseudocolor plot of the Mach number at $T = 0.0$.}
%  \label{fig:t0_mach}
%\end{figure}
 \begin{figure}[htbp]
 \vspace{-4mm}
  \centering
\includegraphics[width = 0.5\linewidth]{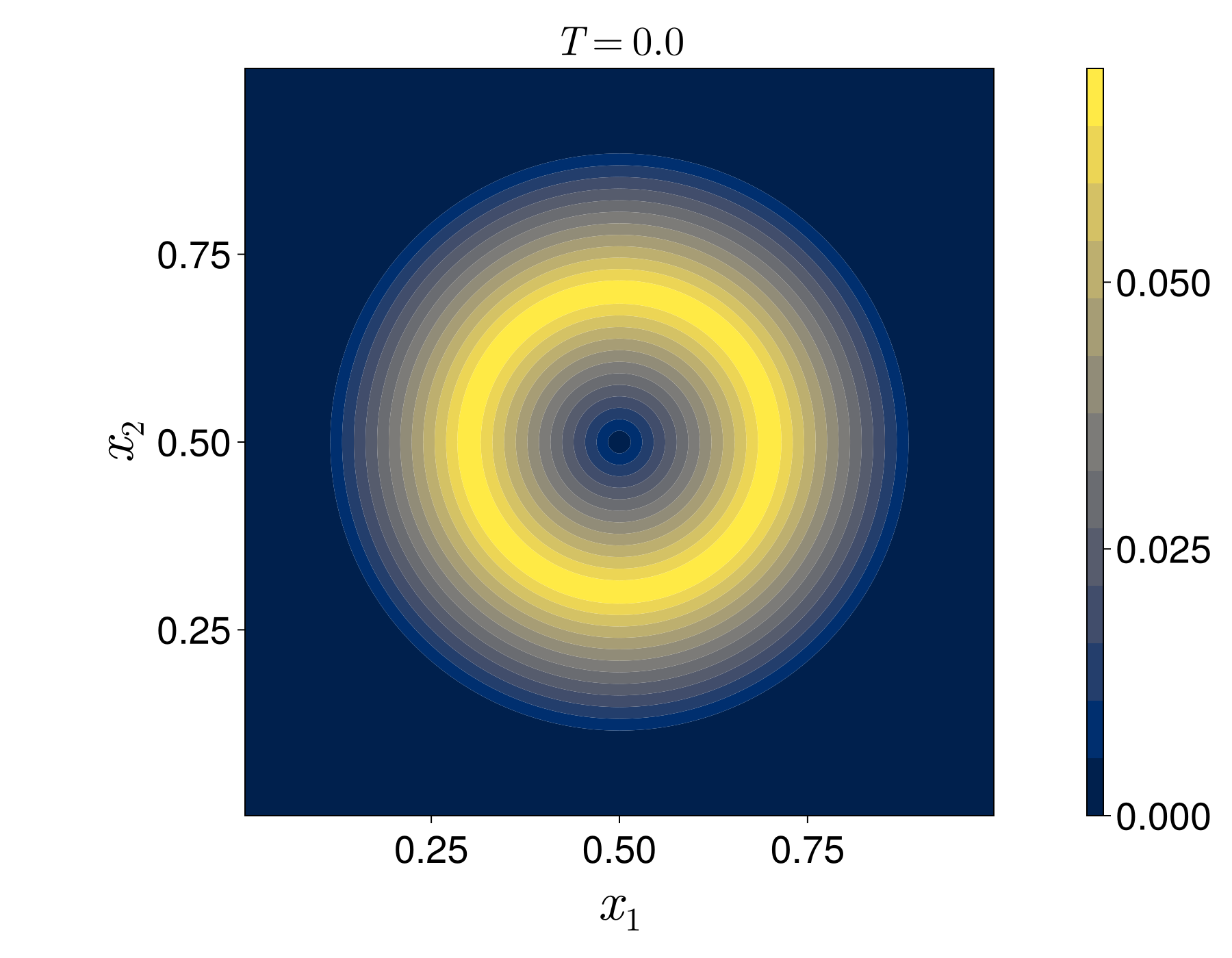}
\vspace{-4mm}
  \caption{2D vortex: pseudocolor plot of the Mach number at $T = 0.0$.}
  \label{fig:t0_mach}
\end{figure}
The Mach number profiles for $\veps \in \{10^{-1}, 10^{-2}, 10^{-3}, 10^{-4}\}$ after one rotation are shown in Figure~\ref{fig:t1_mach}. It conveys that the scheme dissipates minimally and the dissipation is independent of the value of $\veps$.
Figure~\ref{fig:ke_vort} (Left) shows the kinetic energy decay in time for all considered values of $\veps$.
Irrespective of the value of the Mach number the decay is minuscule, i.e., $\mco(10^{-4})$.
Figure~\ref{fig:ke_vort} (Right) shows that the scheme maintains an identical vorticity profile for all values of Mach number.
This showcases the uniform low Mach performance of the newly developed penalized IMEX-WB-AP scheme.
 \begin{figure}[htbp]
 \vspace{-4mm}
  \centering
\includegraphics[width = \linewidth]{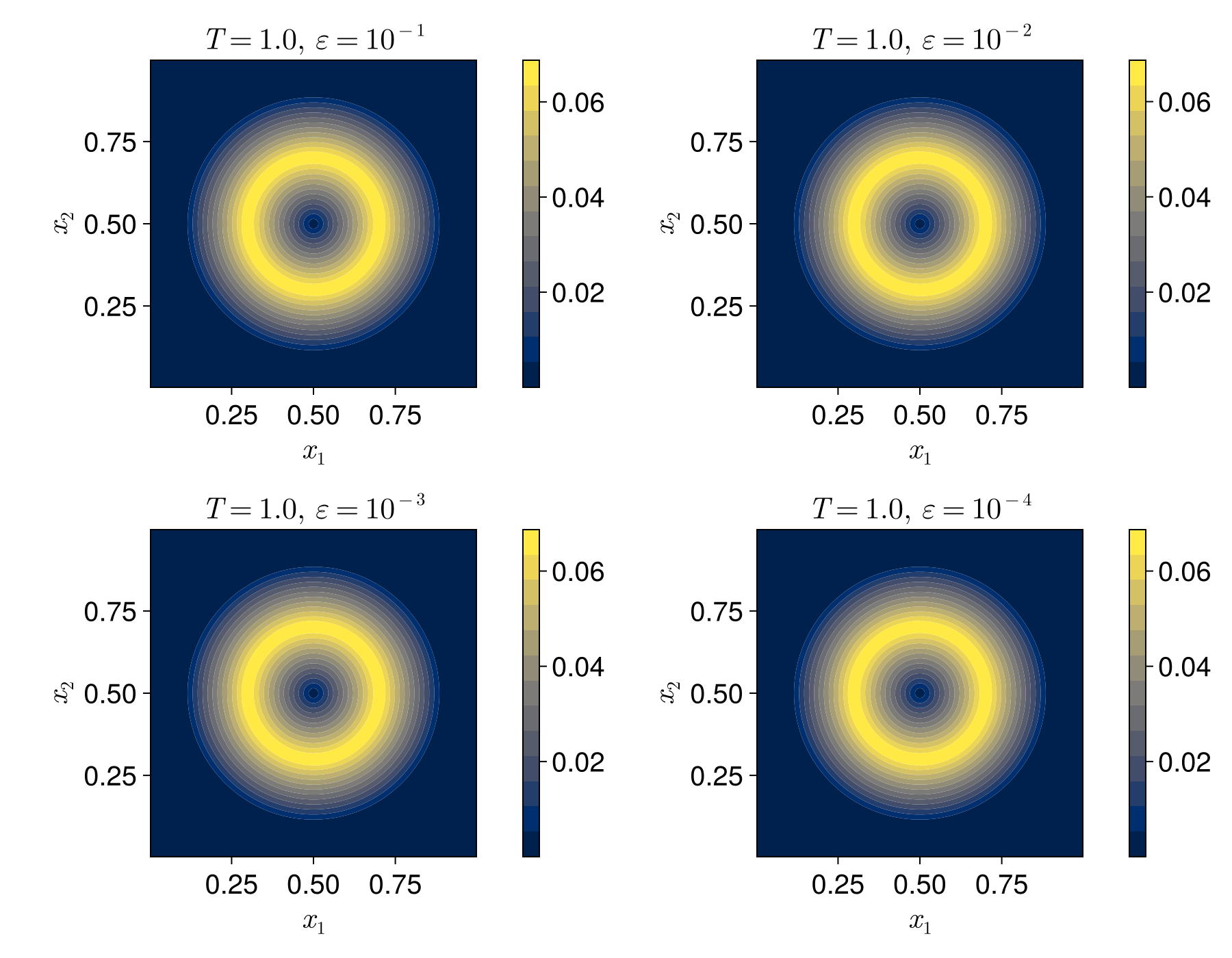}
\vspace{-9mm}
  \caption{2D vortex: pseudocolor plots of Mach number at $T = 1.0$ on $200 \times 200$ mesh, for $\veps = 10^{-1}, 10^{-2}, 10^{-3}$ and $ 10^{-4}$.}
  \label{fig:t1_mach}
\end{figure}
%\begin{figure}[htbp]
%\vspace{-3mm}
%  \centering
%\includegraphics[height=0.23\textheight]{2D_vortex/mach_m1.png}
%\includegraphics[height=0.23\textheight]{2D_vortex/mach_m2.png} \\
%\includegraphics[height=0.23\textheight]{2D_vortex/mach_m3.png}
%\includegraphics[height=0.23\textheight]{2D_vortex/mach_m4.png}
%\vspace{-5mm}
%  \caption{2D vortex: pseudocolor plots of Mach number at $T = 1.0$ on $200 \times 200$ mesh, for $\veps = 10^{-1}, 10^{-2}, 10^{-3}$ and $ 10^{-4}$}
%  \label{fig:t1_mach}
%\end{figure}
%\begin{figure}[htbp]
%  \centering
%\includegraphics[height=0.21\textheight]{2D_vortex/KE_r_decay.png}
%\includegraphics[height=0.22\textheight]{2D_vortex/vorticity.png}
%\vspace{-6mm}
%  \caption{2D vortex: Left: Relative kinetic energy decay; Right: cross section of vorticity}
%  \label{fig:ke_vort}
%\end{figure}
\begin{figure}[htbp]
  \centering
\includegraphics[width = \linewidth]{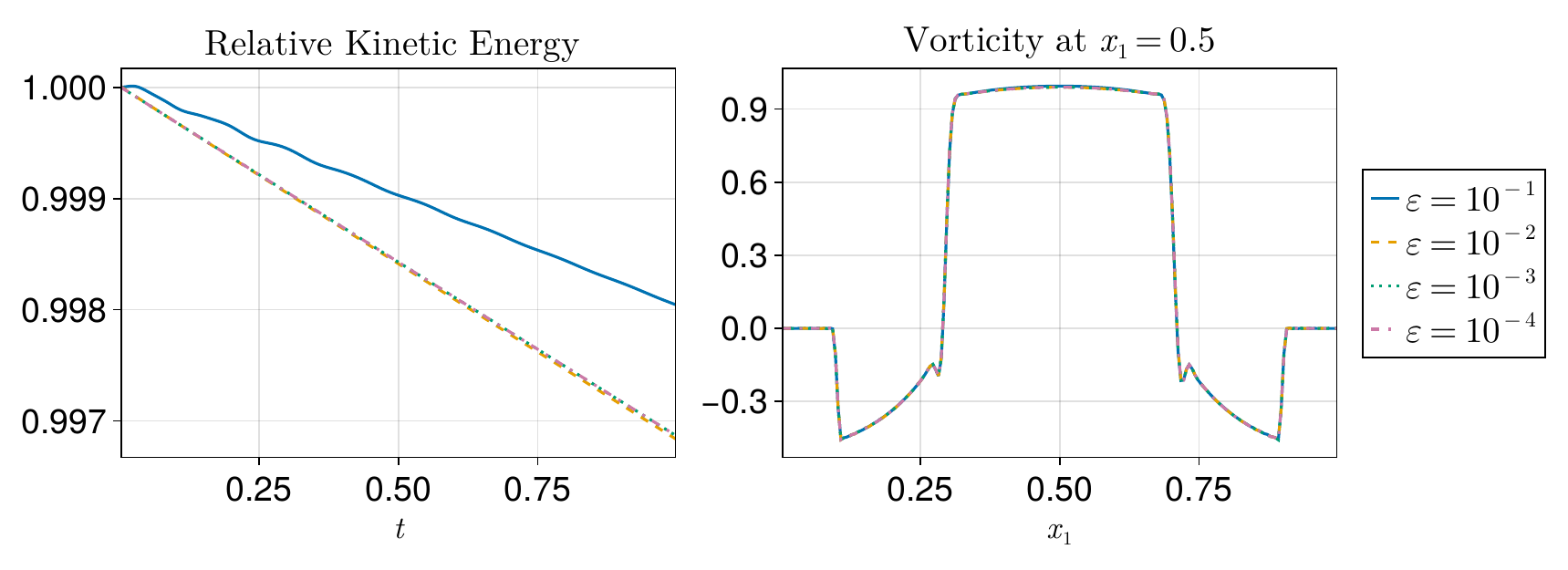}
\vspace{-10mm}
  \caption{2D vortex: Left: Relative kinetic energy decay; Right: cross section of vorticity.}
  \label{fig:ke_vort}
\end{figure}
\section{Conclusions}
\label{sec:conclusion}
In this paper we have designed, analyzed, and tested a higher-order asymptotic-preserving well-balanced scheme for the anelastic limit of the isentropic Euler equations with gravity \eqref{eq:iseuler_mass_ane}--\eqref{eq:iseuler_mom_ane}.
The design leverages the derivation of an equivalent steady-state flux-source pair \eqref{eq:lin_hydrobal} to the non-linear pair \eqref{eq:hydro_bal_an}.
This allows us to penalize the Euler system \eqref{eq:iseuler_mass_ane}--\eqref{eq:iseuler_mom_ane} with a linear flux-source pair leading to the first-order time semi-discrete AP scheme.
The simple structure of the first-order time semi-discrete AP scheme enables usage of the additive IMEX-RK framework \cite{AscherRuuthSpiteri, PR01} to obtain high-order time semi-discretizations.
As demonstrated in the introduction, the AP property is not enough for a scheme to maintain accuracy in the low Mach regime for balance laws such as the Euler system considered here.
The combination of a balance-preserving reconstruction, judicious source-term discretization and a finite-volume approach using Rusanov and central fluxes leads to an AP well-balanced fully-discrete scheme.
Numerical case studies confirm the uniform performance of the designed scheme with respect to the Mach number.

This novel approach of penalizing the governing equations using a linear equilibrium to obtain a linearly implicit scheme serves as a stepping stone to design linearly implicit additive IMEX-RK schemes for more general equations of state in the low Mach number regime.
In this case, for example, if potential temperature variations are to be considered, it will be crucial to understand how to obtain a linear, balance- and asymptotic-preserving relation.
Indeed, this is a work in progress.

\bibliographystyle{siamplain}
\bibliography{references}

@misc{artiano2026asymptoticRepo,
  title={Reproducibility repository for
         "{A}symptotic-Preserving and Well-Balanced Linearly Implicit {IMEX} Schemes
           for the Anelastic Limit of the Isentropic {E}uler Equations with Gravity"},
  author={Artiano, Marco and Ranocha, Hendrik and Samantaray, Saurav},
  year={2026},
  howpublished={\url{https://github.com/MarcoArtiano/2026_asymptotic_preserving_isentropic}},
  doi={10.5281/zenodo.19555933}
}

@book{FN09,
  author    = {Feireisl, Eduard and Novotný, Antonín},
  title     = {Singular Limits in Thermodynamics of Viscous Fluids},
  series    = {Advances in Mathematical Fluid Mechanics},
  publisher = {Birkhäuser/Springer},
  year      = {2009},
  doi       = {10.1007/978-3-7643-8843-0},
  isbn      = {978-3-7643-8843-0},
  address   = {Basel, Switzerland}
}

@article {Kle09,
 Author = {Klein, Rupert},
 Title = {Asymptotics, structure, and integration of sound-proof atmospheric flow equations},
 Journal = {Theor. Comput. Fluid Dyn.},
 Fjournal = {Theoretical and Computational Fluid Dynamics},
 Year = {2009},
 Volume = {23},
 Number = {3},
 Pages = {161--195},
 Doi = {10.1007/s00162-009-0104-y}
}

@article{OP62,
    author = {Ogura, Yoshimitsu and Phillips, Norman A.},
    title = {Scale Analysis of Deep and Shallow Convection in the Atmosphere},
    journal = {J. Atmospheric Sci.},
    fjournal = {Journal of the Atmospheric Sciences},
    volume = {19},
    number = {2},
    pages = {173-179},
    year = {1962},
    issn = {0022-4928},
    doi = {10.1175/1520-0469(1962)019<0173:SAODAS>2.0.CO;2}
}

@article {KM81,
    AUTHOR = {Klainerman, Sergiu and Majda, Andrew J.},
     TITLE = {Singular limits of quasilinear hyperbolic systems with large parameters and the incompressible limit of compressible fluids},
   JOURNAL = {Comm. Pure Appl. Math.},
  FJOURNAL = {Communications on Pure and Applied Mathematics},
    VOLUME = {34},
      YEAR = {1981},
     PAGES = {481--524},
       DOI = {10.1002/cpa.3160340403},
}

@article{FKM19,
author = {Feireisl, Eduard and Klingenberg, Christian and Markfelder, Simon},
title = {On the Low {M}ach Number Limit for the Compressible {E}uler System},
journal = {SIAM Journal on Mathematical Analysis},
volume = {51},
number = {2},
pages = {1496-1513},
year = {2019},
doi = {10.1137/17M1131799}
}

@incollection {AS20_AIMS,
    AUTHOR = {Arun, K. R. and Samantaray, Saurav},
     TITLE = {An asymptotic preserving time integrator for low {M}ach number
              limits of the {E}uler equations with gravity},
 BOOKTITLE = {Hyperbolic problems: theory, numerics, applications},
    SERIES = {AIMS Ser. Appl. Math.},
    VOLUME = {10},
     PAGES = {279--286},
 PUBLISHER = {Am. Inst. Math. Sci. (AIMS), Springfield, MO},
      YEAR = {2020},
      ISBN = {978-1-60133-023-9; 1-60133-023-5},
   MRCLASS = {35L45 (35L67 65M08)},
  MRNUMBER = {4362524},
}

@article {AS20,
    AUTHOR = {Arun, K. R. and Samantaray, Saurav},
     TITLE = {Asymptotic preserving low {M}ach number accurate {IMEX} finite
              volume schemes for the isentropic {E}uler equations},
   JOURNAL = {J. Sci. Comput.},
  FJOURNAL = {Journal of Scientific Computing},
    VOLUME = {82},
      YEAR = {2020},
    NUMBER = {2},
     PAGES = {Art. 35, 32},
      ISSN = {0885-7474},
   MRCLASS = {65M08 (35L45 35L60 35L65 35L67 35Q35 65M06)},
  MRNUMBER = {4059089},
       DOI = {10.1007/s10915-020-01138-8},
}

@article {BQR19,
    AUTHOR = {Boscarino, Sebastiano and Qiu, Jing-Mei and Russo, Giovanni
              and Xiong, Tao},
     TITLE = {A high order semi-implicit {IMEX} {WENO} scheme for the
              all-{M}ach isentropic {E}uler system},
   JOURNAL = {J. Comput. Phys.},
  FJOURNAL = {Journal of Computational Physics},
    VOLUME = {392},
      YEAR = {2019},
     PAGES = {594--618},
      ISSN = {0021-9991},
   MRCLASS = {65M06 (76N15)},
  MRNUMBER = {3950472},
       DOI = {10.1016/j.jcp.2019.04.057},
}

@misc{SAM24,
      title={Asymptotic Preserving Linearly Implicit Additive {IMEX}-{RK} Finite Volume Schemes for Low {M}ach Number Isentropic {E}uler Equations},
      author={Saurav Samantaray},
      year={2024},
      archivePrefix={arXiv},
      primaryClass={math.NA},
      url={https://arxiv.org/abs/2409.05854},
}

@article {Del10,
    AUTHOR = {Dellacherie, St{\'e}phane},
     TITLE = {Analysis of {G}odunov type schemes applied to the compressible
              {E}uler system at low {M}ach number},
   JOURNAL = {J. Comput. Phys.},
  FJOURNAL = {Journal of Computational Physics},
    VOLUME = {229},
      YEAR = {2010},
    NUMBER = {4},
     PAGES = {978--1016},
      ISSN = {0021-9991},
     CODEN = {JCTPAH},
   MRCLASS = {65M06 (76D05)},
  MRNUMBER = {2576236 (2011f:65149)},
       DOI = {10.1016/j.jcp.2009.09.044},
}

@article {DT11,
    AUTHOR = {Degond, Pierre and Tang, Min},
     TITLE = {All speed scheme for the low {M}ach number limit of the
              isentropic {E}uler equations},
   JOURNAL = {Commun. Comput. Phys.},
  FJOURNAL = {Communications in Computational Physics},
    VOLUME = {10},
      YEAR = {2011},
    NUMBER = {1},
     PAGES = {1--31},
      ISSN = {1815-2406,1991-7120},
   MRCLASS = {76M20 (65M06 76L05 76N15)},
  MRNUMBER = {2775032},
MRREVIEWER = {Thomas\ H.\ Sonar},
       DOI = {10.4208/cicp.210709.210610a},
}

@article{DP13,
author = {Dimarco, Giacomo and Pareschi, Lorenzo},
title = {Asymptotic Preserving Implicit-Explicit {R}unge--{K}utta Methods for Nonlinear Kinetic Equations},
journal = {SIAM Journal on Numerical Analysis},
volume = {51},
number = {2},
pages = {1064-1087},
year = {2013},
doi = {10.1137/12087606X}
}

@article{CDS26,
  author  = {Nicolas Crouseilles and Giacomo Dimarco and Saurav Samantaray},
  title   = {High order Asymptotic-Preserving penalized numerical schemes for the {E}uler-{P}oisson system in the quasineutral limit},
  journal = {ESAIM: Mathematical Modelling and Numerical Analysis (ESAIM: M2AN)},
  year    = {2026},
  doi     = {10.1051/m2an/2026004},
  volume = 60,
  number = 2,
  pages = "657-688",
}

@article{NBA14+,
author = {Noelle, S. and Bispen, G. and Arun, K. R. and Luk\'{a}\v{c}ov\'{a}-Medvi\v{d}ov\'{a}, M. and Munz, C.-D.},
title = {A Weakly Asymptotic Preserving Low {M}ach Number Scheme for the {E}uler Equations of Gas Dynamics},
journal = {SIAM Journal on Scientific Computing},
volume = {36},
number = {6},
pages = {B989-B1024},
year = {2014},
doi = {10.1137/120895627}
}

@article{DLV17,
author = {Dimarco, Giacomo and Loub\`{e}re, Rapha\"{e}l and Vignal, Marie-H\'{e}l\`{e}ne},
title = {Study of a New Asymptotic Preserving Scheme for the {E}uler System in the Low {M}ach Number Limit},
journal = {SIAM Journal on Scientific Computing},
volume = {39},
number = {5},
pages = {A2099-A2128},
year = {2017},
doi = {10.1137/16M1069274},
}

@article{Chandrashekar2017,
  author  = {Chandrashekar, Praveen and Zenk, Markus},
  title   = {{Well-Balanced Nodal Discontinuous {G}alerkin Method for {E}uler Equations with Gravity}},
  journal = {Journal of Scientific Computing},
  volume  = {71},
  number  = {3},
  pages   = {1062--1093},
  year    = {2017},
  doi     = {10.1007/s10915-016-0339-x},
}

@article{Chandrashekar2015,
author = {Chandrashekar, Praveen and Klingenberg, Christian},
title = {A Second Order Well-Balanced Finite Volume Scheme for {E}uler Equations with Gravity},
journal = {SIAM Journal on Scientific Computing},
volume = {37},
number = {3},
pages = {B382-B402},
year = {2015},
doi = {10.1137/140984373}
}

@article{BLY17,
title = {Asymptotic preserving {IMEX} finite volume schemes for low {M}ach number {E}uler equations with gravitation},
journal = {Journal of Computational Physics},
volume = {335},
pages = {222-248},
year = {2017},
issn = {0021-9991},
doi = {10.1016/j.jcp.2017.01.020},
author = {Georgij Bispen and Mária Lukáčová-Medvid'ová and Leonid Yelash}
}

@article{ap_jin,
author = {Jin, Shi},
title = {Efficient Asymptotic-Preserving ({AP}) Schemes For Some Multiscale Kinetic Equations},
journal = {SIAM Journal on Scientific Computing},
volume = {21},
number = {2},
pages = {441-454},
year = {1999},
doi = {10.1137/S1064827598334599},
}

@incollection {LB99,
    AUTHOR = {R. J. LeVeque and D. S. Bale},
     TITLE = {Wave propagation methods for conservation laws with source terms},
 BOOKTITLE = {Hyperbolic Problems: Theory, Numerics, Applications, R. Jeltsch and M. Fey,
eds.},
    SERIES = {Internat. Ser. Numer. Math. },
    VOLUME = {130},
     PAGES = {609--618},
 PUBLISHER = {Birkh¨ auser, Basel,},
      YEAR = {1999}
}

@article {AKS22,
    AUTHOR = {Arun, K. R. and Krishnan, Meenakshi and Samantaray, Saurav},
     TITLE = {A unified asymptotic preserving and well-balanced scheme for
              the {E}uler system with multiscale relaxation},
   JOURNAL = {Comput. \& Fluids},
  FJOURNAL = {Computers \& Fluids. An International Journal},
    VOLUME = {233},
      YEAR = {2022},
     PAGES = {Paper No. 105248, 13},
      ISSN = {0045-7930,1879-0747},
   MRCLASS = {65M08 (35L45 35L60 35L65 35L67 76Nxx)},
  MRNUMBER = {4345962},
       DOI = {10.1016/j.compfluid.2021.105248},
}

@article{TPK20,
title = {An all speed second order well-balanced {IMEX} relaxation scheme for the {E}uler equations with gravity},
journal = {Journal of Computational Physics},
volume = {420},
pages = {109723},
year = {2020},
issn = {0021-9991},
doi = {https://doi.org/10.1016/j.jcp.2020.109723},
author = {Andrea Thomann and Gabriella Puppo and Christian Klingenberg},
}

@incollection {PR01,
    AUTHOR = {Pareschi, Lorenzo and Russo, Giovanni},
     TITLE = {Implicit-explicit {R}unge-{K}utta schemes for stiff systems of differential equations},
 BOOKTITLE = {Recent trends in numerical analysis},
    SERIES = {Adv. Theory Comput. Math.},
    VOLUME = {3},
     PAGES = {269--288},
 PUBLISHER = {Nova Sci. Publ., Huntington, NY},
      YEAR = {2001},
   MRCLASS = {65L06 (65M20 82B20)},
  MRNUMBER = {2029975},
MRREVIEWER = {Kazufumi Ozawa}
}

@article {HJL12,
    AUTHOR = {Haack, Jeffrey and Jin, Shi and Liu, Jian-Guo},
     TITLE = {An all-speed asymptotic-preserving method for the isentropic
              {E}uler and {N}avier-{S}tokes equations},
   JOURNAL = {Commun. Comput. Phys.},
  FJOURNAL = {Communications in Computational Physics},
    VOLUME = {12},
      YEAR = {2012},
    NUMBER = {4},
     PAGES = {955--980},
      ISSN = {1815-2406,1991-7120},
   MRCLASS = {76N10 (35Q30 35Q31 65M08 76M12)},
  MRNUMBER = {2913446},
       DOI = {10.4208/cicp.250910.131011a},
}

@article {Jin12,
    AUTHOR = {Jin, Shi},
     TITLE = {Asymptotic preserving ({AP}) schemes for multiscale kinetic
              and hyperbolic equations: a review},
   JOURNAL = {Riv. Math. Univ. Parma (N.S.)},
  FJOURNAL = {Rivista di Matematica della Universit\`a di Parma. New Series. A
              Journal of Pure and Applied Mathematics},
    VOLUME = {3},
      YEAR = {2012},
    NUMBER = {2},
     PAGES = {177--216},
      ISSN = {0035-6298},
   MRCLASS = {82C70 (35Q35 82C80)},
  MRNUMBER = {2964096},
MRREVIEWER = {Eric Sonnendr\"{u}cker},
}

@article{KENNEDY2003139,
title = {Additive {R}unge-{K}utta schemes for convection-diffusion-reaction equations},
journal = {Applied Numerical Mathematics},
volume = {44},
number = {1},
pages = {139-181},
year = {2003},
issn = {0168-9274},
doi = {https://doi.org/10.1016/S0168-9274(02)00138-1},
author = {Christopher A. Kennedy and Mark H. Carpenter}
}

@article{AscherRuuthSpiteri,
    author = {U.M. Ascher, S.J. Ruuth, R.J. Spiteri},
    title = {Implicit–explicit {R}unge–{K}utta methods for time-dependent partial differential equations},
    journal = {Appl. Numer. Math.},
    year = {1997},
pages={151--167}
}

@book {HW96,
    AUTHOR = {Hairer, E. and Wanner, G.},
     TITLE = {Solving ordinary differential equations. {II}},
    SERIES = {Springer Series in Computational Mathematics},
    VOLUME = {14},
   EDITION = {Second},
      NOTE = {Stiff and differential-algebraic problems},
 PUBLISHER = {Springer-Verlag, Berlin},
      YEAR = {1996},
     PAGES = {xvi+614},
      ISBN = {3-540-60452-9},
   MRCLASS = {65-02 (34A09 34A45 65-01 65Lxx)},
  MRNUMBER = {1439506},
       DOI = {10.1007/978-3-642-05221-7},
}

@article {Bos07,
    AUTHOR = {Boscarino, Sebastiano},
     TITLE = {Error analysis of {IMEX} {R}unge-{K}utta methods derived from
              differential-algebraic systems},
   JOURNAL = {SIAM J. Numer. Anal.},
  FJOURNAL = {SIAM Journal on Numerical Analysis},
    VOLUME = {45},
      YEAR = {2007},
    NUMBER = {4},
     PAGES = {1600--1621},
      ISSN = {0036-1429,1095-7170},
   MRCLASS = {65L06 (34E05)},
  MRNUMBER = {2338401},
MRREVIEWER = {Brian\ Bradie},
       DOI = {10.1137/060656929},
}

@article{boscarino,
author = {Boscarino, S. and Pareschi, L. and Russo, G.},
title = {{Implicit-Explicit Runge--Kutta schemes for hyperbolic systems and kinetic equations in the diffusion limit}},
journal = {SIAM J. Sci. Comput.},
volume = {35},
number = {1},
pages = {A22-A51},
year = {2013}
}

@article{BR13,
author = {Boscarino, Sebastiano and Russo, Giovanni},
title = {Flux-Explicit {IMEX} {R}unge--{K}utta Schemes for Hyperbolic to Parabolic Relaxation Problems},
journal = {SIAM Journal on Numerical Analysis},
volume = {51},
number = {1},
pages = {163-190},
year = {2013},
doi = {10.1137/110850803},
}

@article {GP16,
    AUTHOR = {Guermond, Jean-Luc and Popov, Bojan},
     TITLE = {Fast estimation from above of the maximum wave speed in the
              {R}iemann problem for the {E}uler equations},
   JOURNAL = {J. Comput. Phys.},
  FJOURNAL = {Journal of Computational Physics},
    VOLUME = {321},
      YEAR = {2016},
     PAGES = {908--926},
      ISSN = {0021-9991,1090-2716},
   MRCLASS = {76N15 (35Q20)},
  MRNUMBER = {3527597},
MRREVIEWER = {Bernard\ Ducomet},
       DOI = {10.1016/j.jcp.2016.05.054},
}

@article{SBG+99,
title = {Extension of Finite Volume Compressible Flow Solvers to Multi-dimensional, Variable Density Zero {M}ach Number Flows},
journal = {Journal of Computational Physics},
volume = {155},
number = {2},
pages = {248-286},
year = {1999},
issn = {0021-9991},
doi = {https://doi.org/10.1006/jcph.1999.6327},
author = {T. Schneider and N. Botta and K.J. Geratz and R. Klein},
}

@phdthesis{Zenk2018,
  author      = {Markus Zenk},
  title       = {On Numerical Methods for Astrophysical Applications},
  type        = {PhD thesis},
  school      = {Universit{\"a}t W{\"u}rzburg},
  year        = {2018},
}
\end{document}